\documentclass[envcountsect,envcountsame]{svjour3}
\usepackage{color}
\usepackage{amsmath,amssymb,enumitem,graphicx,xspace,makecell,pifont,hyperref}
\usepackage[normalem]{ulem}
\usepackage[all]{xy}
\smartqed
\newcommand{\cmark}{\ding{51}}
\newcommand{\xmark}{\ding{55}}
\newcommand{\lldash}{\rotatebox{55}{$\vdash$}}
\newcommand{\rrdash}{\rotatebox{120}{$\vdash$}}
\newcommand{\cat}[1]{\ensuremath{\mathbf{#1}}}
\newcommand{\p}{\ensuremath{{}_{\bullet}}}
\newcommand{\GLat}{\cat{GLat}}
\newcommand{\FVect}{\cat{FVect}}
\newcommand{\MVect}{\cat{MVect}}
\newcommand{\Matr}{\cat{Matr}}
\newcommand{\SMatr}{\cat{SMatr}}
\newcommand{\FMatr}{\cat{FMatr}}
\newcommand{\LMatr}{\cat{LMatr}}
\newcommand{\Matrp}{\cat{Matr_\bullet}}
\newcommand{\SMatrp}{\cat{SMatr_\bullet}}
\newcommand{\FMatrp}{\cat{FMatr_\bullet}}
\newcommand{\LMatrp}{\cat{LMatr_\bullet}}
\newcommand{\I}{\ensuremath{\mathcal{I}}}
\newcommand{\F}{\ensuremath{\mathcal{F}}}
\newcommand{\id}[1]{\ensuremath{\mathrm{id}_{#1}}}
\newcommand{\Mat}{\ensuremath{M}}
\DeclareMathOperator{\si}{si}
\DeclareMathOperator{\rank}{rk}
\DeclareMathOperator{\clos}{cl}
\DeclareMathOperator{\card}{\#\hspace*{-.5mm}}
\DeclareMathOperator{\supp}{supp}
\newcommand{\sip}{\ensuremath{\si_{\bullet}}}

\newcommand{\ie}{\textit{i.e.}\xspace}

\title{The category of matroids}
\author{Chris Heunen \and Vaia Patta}
\institute{C. Heunen \at University of Edinburgh. \\ \email{chris.heunen@ed.ac.uk} \\ Supported by EPSRC Fellowship EP/L002388/1. 
 \and V. Patta \at University of Oxford. \\ \email{vaia.patta@cs.ox.ac.uk} \\ Supported by EPSRC Grant EP/K503113/1. }
\date{\today}

\begin{document}
\maketitle
        
\begin{abstract}
  The structure of the category of matroids and strong maps is investigated: it has coproducts and equalizers, but not products or coequalizers; there are functors from the categories of graphs and vector spaces, the latter being faithful; there is a functor to the category of geometric lattices, that is nearly full; there are various adjunctions and free constructions on subcategories, inducing a simplification monad; there are two orthogonal factorization systems; some, but not many, combinatorial constructions from matroid theory are functorial. Finally, a characterization of matroids in terms of optimality of the greedy algorithm can be rephrased in terms of limits.
\end{abstract}

\section{Introduction}\label{sec:intro}

Matroids are very general structures that capture the notion of independence. They were introduced by Whitney in 1935 as a common generalization of independence in linear algebra and independence in graph theory~\cite{whitney:matroids}. They were linked to projective geometry by Mac Lane shortly thereafter~\cite{maclane:projectivegeometry}, and have found a great many applications in geometry, topology, combinatorial optimization, network theory, and coding theory since~\cite{white:matroids,oxley:matroids}. However, matroids have mostly been studied individually, and most authors do not consider maps between matroids. The only explicitly categorical work we are aware of is~\cite{alhawary:thesis,alhawary:free,alhawarymcrae:axiomatic,anderson1999matroid,lu2010categorical,stamps2013topological}. The purpose of this article is to survey the properties of the category of matroids.

Of course, speaking of ``the'' category of matroids entails a choice of morphisms. There are compelling reasons to choose so-called strong maps:
\begin{itemize}
  \item They are a natural choice of structure-preserving functions.
  \item The construction making matroids out of vector spaces is then functorial, as we will show in Section~\ref{sec:functors}.
  \item There is then an elegant orthogonal factorization system, as we will show in Section~\ref{sec:factorization}, that connects to the important topic of minors, as we will show in Section~\ref{sec:deletioncontraction}.
\end{itemize}
The resulting category has some limits and colimits, but not all, as we will show in Section~\ref{sec:limits}, see Figure~\ref{fig:overview}. It will also turn out that many combinatorial constructions from matroid theory are not even functorial, as we will show in Section~\ref{sec:constructions}, see Figure~\ref{fig:overview}. Weaker choices of morphisms, such as the so-called weak maps and comaps~\cite{white:matroids}, lead to less well-behaved categories.
However, as we show in Section~\ref{sec:greedy}, strong maps allow us to characterize matroids in terms of optimality of the greedy algorithm via limits.

To complete the overview of this article: Section~\ref{sec:category} recalls the various equivalent definitions of the objects of the category that we will need, and Section~\ref{sec:adjunctions} establishes adjunctions between various subcategories of well-studied types of matroids, showing that simplification is monadic. See Figure~\ref{fig:adjunctions}.

One feature of matroid theory that we leave for future work is duality: functoriality of this construction needs a choice of morphisms that stands to strong maps as relations stand to functions. Finally, a variation on the concept of matroid called bimatroids~\cite{kung1978bimatroids} have a distinctly 2-categorical flavour to them, that we do not go in to here.

\begin{figure}
  \centering
  \begin{tabular}{|l|c|}\hline
  (Co)limit & present \\
  \hline
  Products & \xmark \\
  Equalizers & \cmark \\
  Pullbacks & \xmark \\
  Coproducts & \cmark \\
  Coequalizers & \xmark \\
  Pushouts & \xmark \\
  Exponentials & \xmark \\ \hline
  \end{tabular}
  \hfill
  \begin{tabular}{|l|c|}\hline
    Construction & functorial \\
    \hline
    Duality & \xmark \\
    Addition & \cmark \\
    Deletion & \cmark \\
    Contraction & \cmark \\
    Minor & \cmark \\
    Extension & \xmark \\
    Truncation & \xmark \\
    Erection & \xmark \\
    \hline
  \end{tabular}
  \hfill
  \begin{tabular}{|l|c|}\hline
    Construction & monoidal \\
    \hline
    Union & \xmark \\
    Intersection & \xmark \\
    Half-dual union & \xmark \\
    Intertwining & \xmark \\
    Parallel connection & \cmark \\
    Series connection & \cmark \\
    \hline
  \end{tabular}
  \caption{Overview of the categorical nature of matroids}      
  \label{fig:overview}
\end{figure}

\begin{figure}
  \[\xymatrix@C+10ex@R+5ex{
    *+++++[d]{\mbox{Matroids}}
    \ar@{}|-{\perp}[r]
    \ar@<-1ex>@{<-^{)}}[r]
    \ar@<-1ex>[d]
    \ar[dr]
    \ar@<-1ex>@{}|(.55){\lldash}[dr]
    \ar@<-4ex>[dr]
    & 
    *+++++[d]{\mbox{Loopless matroids}}
    \ar@<-1ex>@{<-}[l]
    \ar@<1ex>[d]
    \ar@<-2ex>[dl]
    \\
    *++++++[d]{\mbox{Simple matroids}}
    \ar@{}|-{\perp}[r]
    \ar@<-1ex>[r]
    \ar@<-1ex>@{^{(}->}[u]
    \ar@{}|-{\dashv}[u]
    \ar@{^{(}->}[ur]
    &
    *++++++[d]{\mbox{Free matroids}}
    \ar@<-1ex>@{^{(}->}[l]
    \ar@{}|-{\dashv}[u]
    \ar@<1ex>@{_{(}->}[u]
    \ar@<-2ex>@{_{(}->}[ul]
    \ar@<2ex>[ul]
  }\]
  \caption{Overview of adjunctions involving the category of matroids}
  \label{fig:adjunctions}
\end{figure}
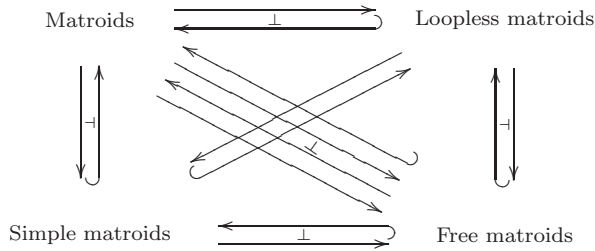

\section{The category}\label{sec:category}

Matroids, the objects of our category of interest, can be defined in many equivalent ways. We list some that are used later, writing $\card X$ for the cardinality of a set $X$. Although some of the theory of matroids goes through for infinite sets~\cite{bruhnetal:infinitematroids}, that situation is much more intricate. Therefore we will only consider the finite case.

\begin{definition}\label{def:matroid}
  A \emph{matroid} $M$ consists of a finite \emph{ground set} $|M|$ with, equivalently:
  \begin{itemize}
  \item a family of $\mathcal{I}$ of subsets of $|M|$, called the \emph{independent sets}, satisfying:
    \begin{itemize}
      \item nontrivial: the empty set is independent;
      \item downward closed: if $I\in\mathcal{I}$ and $J \subseteq I$ then also $J\in\mathcal{I}$;
      \item \emph{independence augmentation}: if $I,J\in\mathcal{I}$ and $\card I<\card J$, then $I\cup\{e\}\in\mathcal{I}$ for some $e\in J\setminus I$.
    \end{itemize}
    a maximal independent set is called a \emph{basis}, and their collection is denoted $\mathcal{B}$;
  \item a family $\mathcal{F}$ of subsets of $|M|$, called the \emph{closed sets} or \emph{flats}, satisfying:
    \begin{itemize}
      \item nontrivial: $|M|$ itself is closed;
      \item closed under intersection: if $F,G \in \mathcal{F}$ then also $F \cap G \in \mathcal{F}$;
      \item \emph{partitioning}: if $\{F_1,F_2,\ldots\}$ are the minimal flats properly containing a flat $F$, then  $\{F_1\setminus F,F_2 \setminus F, \ldots \}$ partitions $|M| \setminus F$.      
    \end{itemize}
  \item a \emph{rank function} $\rank \colon 2^{|M|} \to \mathbb{N}$, satisfying:
    \begin{itemize}
      \item bounded: $0 \leq r(X) \leq \card |M|$ for all $X \subseteq |M|$;
      \item monotonic: if $X\subseteq Y\subseteq |M|$, then $\rank(X)\leq \rank(Y)$;
      \item valuation: if $X,Y\subseteq |M|$, then $\rank (X \cup Y) + r(X \cap Y) \leq \rank(X) + \rank(Y)$.
    \end{itemize}
  \end{itemize}
\end{definition}

To see how to turn the data of one of the equivalent definitions into the data of another, define the \emph{closure} operation $\clos \colon 2^{|M|} \to 2^{|M|}$ by 
\[
  \clos(X) = \big\{x\in |M| \;\big|\; \rank(X\cup \{x\})=\rank(X) \big\}.
\]
A closed set or flat is then a subset of $|M|$ which equals its closure, and $\rank(X)$ is the size of the largest independent set contained in $X \subseteq |M|$.
Maximal flats properly contained in $|M|$ are also called \emph{hyperplanes}; their collection is denoted by $\mathcal{H}$.

Some elements of a matroid are of particular interest.

\begin{definition}
  A \emph{loop} is an element of a matroid that is not contained in any independent set, or equivalently, an element that is contained in all flats. An \emph{isthmus} is an element that is included in every basis. Nonloop elements of the same rank-1 flat are called \emph{parallel}. 
\end{definition}

The following special types of matroids are of special interest in matroid theory.

\begin{definition}
  A matroid is \emph{pointed} when it has a distinguished loop, denoted $\bullet$ and called \emph{the point}.
  A (pointed) matroid is:
  \begin{itemize}
    \item \emph{loopless} when it has no loops (other than the point);
    \item \emph{simple} when it has no loops (other than the point) or parallel elements;
    \item \emph{free} when every subset (not containing the point) is independent.
  \end{itemize}
\end{definition}

Here are some typical examples.

\begin{example}\label{ex:vectorspace}
  Any finite vector space $V$ gives rise to a matroid $\Mat(V)$, whose ground set is $V$, and whose independent sets are those subsets of $V$ that are linearly independent; flats correspond to vector subspaces of $V$, the closure operation takes linear spans, and the rank function computes the dimension of the linear span. The matroid $\Mat(V)$ is pointed when choosing the distinguished loop $\bullet$ to be 0. The pointed matroid $\Mat(V)$ is free only when $V$ is zero-dimensional or when $V=\mathbb{Z}_2$.
  (In fact, infinite vector spaces also satisfy the axioms for independence, but we will not consider the infinite setting here.)
\end{example}

\begin{example}\label{ex:graph}
  Any undirected multigraph $G$ gives rise to a matroid $M(G)$, whose ground set consists of the edges, and where a subset is independent when it contains no cycles. Loops of $M(G)$ are precisely edges of $G$ between a vertex and itself. Isthmuses are edges that are not contained in any cycle. Parallel elements of $M(G)$ are precisely parallel edges of $G$.
  We can point $M(G)$ by choosing a loop. The matroid $M(G)$ is simple when it has no loops and no parallel edges. It is free only when $G$ has no cycles.
\end{example}

\begin{example}\label{ex:lattice}
  We can specify a matroid $M$ by giving (the Hasse diagram of) its partially ordered set $L(M)$ of flats. For example:
  
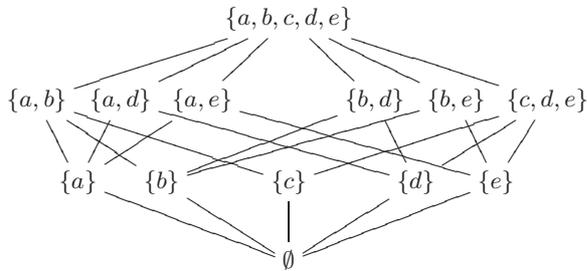
\[\xymatrix@R-2ex@C-8ex{
    &&&&& \{a,b,c,d,e\} \\
    \{a,b\} \ar@{-}[urrrrr] && \{a,d\} \ar@{-}[urrr] && \{a,e\} \ar@{-}[ur] && \{b,d\} \ar@{-}[ul] && \{b,e\} \ar@{-}[ulll] && \{c,d,e\} \ar@{-}[ulllll]\\
    & \{a\} \ar@{-}[ul] \ar@{-}[ur] \ar@{-}[urrr] && \{b\} \ar@{-}[ulll] \ar@{-}[urrr] \ar@{-}[urrrrr] && \{c\} \ar@{-}[ulllll] \ar@{-}[urrrrr] && \{d\} \ar@{-}[ul] \ar@{-}[urrr] \ar@{-}[ulllll] && \{e\} \ar@{-}[ur] \ar@{-}[ul] \ar@{-}[ulllll] \\
    &&&&& \emptyset \ar@{-}[u] \ar@{-}[ull] \ar@{-}[ullll] \ar@{-}[urr] \ar@{-}[urrrr]
  }\]
  is a matroid with ground set $\{a,b,c,d,e\}$.
  As we will see later, any geometric lattice $L$ gives rise to a matroid $M(L)$ this way. 
  The matroid $M(L)$ is simple. It is free only when $L$ is a full powerset lattice.
\end{example}

Matroids form a category with strong maps between them.

\begin{definition}\label{def:strongmap}
  A \emph{strong map} from $M$ to $N$ is a function $f \colon |M| \to |N|$ such that the inverse image of any flat in $N$ is a flat in $M$. Write $\Matr$ for the category of matroids and strong maps, and $\LMatr$, $\SMatr$, $\FMatr$ for the full subcategories of loopless, simple, and free matroids.

  A strong map between pointed matroids is \emph{pointed} when it sends the point to the point. Write $\Matrp$ for the category of pointed matroids and pointed strong maps, and $\LMatrp$, $\SMatrp$, $\FMatrp$ for the full subcategories of loopless, simple, and free matroids.
\end{definition}

The flats of a matroid $M$, when ordered by inclusion, form a geometric lattice $L(M)$, where the height of an element is the rank of the corresponding flat; and conversely, every geometric lattice is the lattice of some matroid~\cite[Theorem~1.7.5]{oxley:matroids}. This gives another way to think about strong maps between matroids.

\begin{lemma}\label{lem:latticemaps}
  For $M,N\in\Matr$ and $f \colon |M| \to |N|$ the following are equivalent:
  \begin{enumerate}[label=(\alph*)]
    \item $f$ is a strong map;
    \item $\rank(f(Y))-\rank(f(X))\leq \rank(Y)-\rank(X)$ for all $X\subseteq Y\subseteq |M|$;
    \item the function $L(f) \colon L(M) \to L(N)$ given by $X \mapsto \clos(f(X))$ preserves joins and sends elements of height 1 to elements of height 0 or 1.
  \end{enumerate}
\end{lemma}
\begin{proof}
  See~\cite[Propositions~8.1.3 and~8.1.6]{white:matroids}.
  \qed
\end{proof}

The assignment $M \mapsto |M|$ forms a forgetful functor $|-| \colon \Matr \to \cat{FinSet}$ from the category of matroids to the category $\cat{FinSet}$ of finite sets and functions.

\begin{theorem}\label{thm:finset}
  There is a series of adjunctions $F\dashv|-|\dashv C \dashv (-)_0$ given by
  \begin{align*}
    \mathcal{F}_{F(X)} & = 2^X, & F(f) & = f\text{,} \\
    \mathcal{F}_{C(X)} & = \{X\}, & C(f) & = f\text{,} \\
    (M)_0 &= F_0, & (f)_0 & = f|_{F_0} \text{.}
  \end{align*}
  where $F_0$ denotes the bottom element of $L(M)$. There are no further adjoints.
\end{theorem}
\begin{proof}
  Functoriality is clear. To see that $F \dashv |-|$, note that for every set $X$, every matroid $M$, and every function $f \colon X\to|M|$, there must exist a function $\eta_X \colon X\to X$ and a unique strong map $\hat f \colon F(X)\to M$ satisfying $f=\hat f\circ\eta_X$, namely $\eta_X=\id{X}$ and $\hat f=f$.

  Similarly, to see that $|-| \dashv C$, observe that for every strong map $f \colon M\to C(X)$ there must exist a strong map $\eta_M \colon M\to C(|M|)$ and a unique function $\hat f \colon |M|\to X$ satisfying $f=\hat f\circ\eta_M$, namely $\eta_M(x)=x$ and $\hat f=f$.
  That $C \dashv (-)_0$ is seen similarly: set $\eta_X=\id{X}$ and $\hat f=f$.

  Now suppose $K$ is left adjoint to $F$. 
  Let $M$ be the (pointed) matroid with flats
  \begin{align*}
        \big\{ & \{l\},\{l,a\},\{l,b\},\{l,c\},\{l,d\},\\
        & \{l,a,b\},\{l,a,c\},\{l,a,d\},\{l,b,c\},\{l,b,d\},\{l,c,d\},\{l,a,b,c,d\}\big\}.
  \end{align*}
  We will count the morphisms $K(M) \to X$ and $M \to F(X)$.
  Because the subset $\eta(M) \subseteq K(M)$ has at most 4 elements, we may assume that $\card K(M) \leq 4$.
  If $\card K(M) \leq 3$, pick $\card X=2$ for fewer maps $K(M)\to X$ than $M\to F(X)$. 
  If $\card K(M) = 4$, pick $\card X=4$. Since the unit $\eta$ of the adjunction cannot map nonloops to loops, it must map nonloops $a\neq b$ of $M$ to the same nonloop of $K(M)$. But this prevents surjective $f \colon M \to F(X)$ mapping $a,b$ to the same nonloop from having a mate $\hat{f} \colon K(M)\to X$. Thus $F$ has no left adjoint.

  Finally, suppose $B$ were a right adjoint to $(-)_0$. Let $M$ be a nonempty loopless matroid. Then there is a unique function $(M)_0 \to X$, so there can only exist one strong map $f\colon M\to B(X)$. Hence $\card B(X)=1$ for all $X$. But this cannot be a right adjoint to $(-)_0$, for if $M$ instead is a matroid with at least one loop and $\card X>1$, there are multiple functions $(M)_0 \to X$.
  \qed
\end{proof}

Matroids of the form $F(X)$ for some finite set $X$ are precisely free matroids. We will call matroids of the form $C(X)$ \emph{cofree} matroids.

The fibre of $|-|$ over any finite set $X$ is partially ordered: if $M$ and $N$ are matroids with $|M|=|N|=X$, then $M\le N$ if and only if $\mathcal{F}_M\subseteq\mathcal{F}_N$. This resembles the situation in general topology, with $\le$ indicating a ``finer'' matroid structure, $F(X)$ being the finest (most closed sets) one, and $C(X)$ being the coarsest (fewest closed sets) one.

\section{Limits and colimits}\label{sec:limits}

We now examine limits and colimits in $\Matr$ and its subcategories. We give proofs and counterexamples to accommodate different variations and to repair mistakes in the literature.

\begin{remark}\label{rem:constructlimit}
  In some ways, including the computation of limits and colimits, the category of matroids is analogous to the category of topological spaces and continuous functions.

  Let \cat{D} be a diagram in $\Matr$. To construct its limit (if it exists) first take the limit $L$ of $|\cat{D}|$ in $\cat{FinSet}$ and denote the limit cone by $\lambda_X \colon L\to|X|$ for $X\in\cat{D}$. Then the limit of \cat{D} exists if and only if there is a coarsest matroid structure on $L$ making the $\lambda_X$ strong, that is, if and only if there is a coarsest matroid structure $M$ on $L$ such that $\{\lambda_X^{-1}(S) \mid X\in\cat{D},S\in\mathcal{F}_X\}\subseteq\mathcal{F}_M$. 

  Similarly, for the colimit of \cat{D}, take the colimit $K$ in \cat{FinSet} and form a colimit cocone $\kappa_X \colon |X|\to K$ in \cat{FinSet}. Then the colimit exists if and only if there is a finest matroid structure $N$ on $K$ such that $\mathcal{F}_N\subseteq\{S\subseteq K \mid \forall X\in\cat{D},\kappa_X^{-1}(S)\in\mathcal{F}_X\}$. 
\end{remark}

Clearly the empty matroid is an initial object in all subcategories of matroids we consider.
The one-element matroid where the element is a loop is a terminal object in all categories of matroids we consider.

The functor $|-| \colon \Matr \to \cat{FinSet}$ restricts to an isomorphism of categories $\FMatr \to \cat{FinSet}$. It follows that $\cat{FMatr}$ has all finite (co)limits. We turn our attention to larger matroid categories, starting with coproducts and equalizers, that are known to exist, but may alternatively be described as in Remark~\ref{rem:constructlimit}.

\begin{proposition}\label{prop:coproducts}
  The categories $\SMatr$, $\LMatr$, and $\Matr$ have coproducts.
\end{proposition}
\begin{proof}
  See~\cite{brylawski1971combinatorial} and also~\cite[Proposition~4]{crapo1971constructions}.
  The coproduct $M+N$ has ground set $|M+N|=|M| \sqcup |N|$ and flats $\{F \cup G \mid F \in \mathcal{F}_M, G \in \mathcal{F}_N \}$. 
  It is easy to see that if $M$ and $N$ are 
  simple or loopless, then so is $M+N$.
  The coprojections are the inclusions $M \to M+N$ and $N \to M+N$, and it is clear that there is a unique strong cotuple $M+N \to P$ of strong maps $M \to P$ and $N \to P$.
  \qed
\end{proof}

\begin{proposition}\label{prop:equalizers}
  The categories $\SMatr$, $\LMatr$, and $\Matr$ have equalizers.
\end{proposition}
\begin{proof}
  See also~\cite[Theorem~53]{alhawary:thesis}. 
  The equalizer of $f,g \colon M \to N$ is the inclusion of the matroid $E$ with $|E|=\{x \in |M| \mid f(x)=g(x)\}$ and $\mathcal{F}_E = \{ F \cap |E| \mid F \in \mathcal{F}_M\}$.
  This is a well-defined matroid, that clearly satisfies the universal property.
  Finally, if $M$ is simple or loopless, then so is $E$.
  \qed
\end{proof}

\begin{remark}
  In contrast to topological spaces, there is an obstacle to the existence of all finite limits and colimits of matroids: matroid flat structures on a set $X$ are not closed under finite intersections in $2^{2^X}$ because of the partition property. 
  Thus the category of ``generalized matroids'', with objects defined via closed subsets by removing the partition axiom and strong maps as morphisms, gives a finitely complete and cocomplete category containing $\Matr$. The inclusion preserves coproducts and equalizers. We will now see that products and coequalizers are not reflected.
\end{remark}

\begin{proposition}\label{prop:products}
  The categories $\SMatr$, $\LMatr$, and $\Matr$ do not have all products.
\end{proposition}
\begin{proof}
  See~\cite[Proposition~5]{crapo1971constructions}. 
  Consider the simple matroid $M$ of rank 2 with $|M|=\{a,b,c,d\}$ and suppose $M \times M$ existed in $\SMatr$.
  A pair $(x,y) \in |M| \times |M|$ induces two strong maps $x,y \colon 1 \to M$, and hence a unique tuple $(x,y) \colon 1 \to M \times M$ with $\pi_1(x,y)=x$ and $\pi_2(x,y)=y$. Hence $M \times M$ has at least 16 elements.
  The tuple of any permutation $\sigma \colon |M| \to |M|$ with the identity gives a strong map $M \to M \times M$, whose image $\{(a,\sigma(a)), (b,\sigma(b)), (c,\sigma(c)), (d,\sigma(d))\}$ must be a flat.
  But $\{(a,a),(b,b),(c,c),(d,d)\}$, $\{(a,a),(b,b),(c,d),(d,c)\}$ cannot both be flats.%
  \qed
\end{proof}

It follows that these categories do not have pullbacks or exponentials either.

\begin{proposition}\label{prop:pushouts}
  The categories $\SMatr$, $\LMatr$, and $\Matr$ do not have all pushouts, but pushouts under cofree matroids exist.
\end{proposition}
\begin{proof}
  See~\cite[Proposition~7]{crapo1971constructions}.
  Consider the simple matroid $M$ with six elements where every pair forms a hyperplane.
  Let $M_1$ be the extension of $M$ by one element so that two of the hyperplanes are triples, and let $M_2$ be the extension of $M$ by one element so that three of the hyperplanes are triples, two of which coincide with $M_1$.
  Suppose the pushout $P$ of the embeddings $f_1 \colon M_1 \to M$ and $f_2 \colon M_2 \to M$ existed.
  Then the image of $M$ colimit cone $P'$ cannot be of rank 3, because then $M_1$ and $M_2$ would be embedded in $P$ with overlap $M$.
  The flats $\{a,b\}$ and $\{d,e\}$ meet in a point of $P$ which is not in the flat $\clos\{e,f\}$ of $M_1$ but is in the flat $\clos\{e,f\}$ in $M_2$. Therefore the image of $M$ in $P$ has rank at most 1.
  But there exist different choices $P'$, such that no rank-1 flat has a strong map to both choices $P'$.

  Morphisms $f \colon C(X) \to M$ and $g \colon C(X) \to N$ in $\Matr$ correspond to $\hat{f} \colon X \to |M|$ and $\hat{g} \colon X \to |N|$ in $\cat{FinSet}$ by Theorem~\ref{thm:finset}. These have a pushout $Y$ in \cat{FinSet}, and it follows that $C(Y)$ is the pushout of $f$ and $g$ in $\Matr$.
  \qed  
\end{proof}

\begin{proposition}\label{prop:coequalizers}
  The categories $\SMatr$, $\LMatr$, and $\Matr$ do not have coequalizers.
\end{proposition}
\begin{proof}
  Consider the following objects and morphisms:
  \begin{align*}
    |M| & = |N| = \{\bullet,1,2,3,4\}, \\
    \mathcal{H}_M & = \mathcal{H}_N = \big\{\{\bullet,1,2\},\{\bullet,1,3\},\{\bullet,2,3\},\{\bullet,1,4\},\{\bullet,2,4\},\{\bullet,3,4\}\big\}, \\
    f & = \id,\\
    g & = \big( \bullet\mapsto\bullet,1\mapsto2,2\mapsto1,3\mapsto3,4\mapsto4 \big).
  \end{align*}
  If the coequalizer $c \colon N \to C$ of $f,g \colon M \to N$ existed,
  then~\cite[Theorem~54]{alhawary:thesis}\footnote{That proof is invalid, as the map $y \mapsto [y]$ need not be strong.
  However, it does hold for the choices in our proof below; that is, there are strong maps from $N$ to options~\eqref{eq:c1} and~\eqref{eq:c3}.
  }
   its ground set must be (in bijection with) the quotient $|N|/\mathop{\sim}$ by the equivalence relation generated by $f(x) \sim g(x)$ for all $x \in |N|$, and $c$ must map $x$ to its equivalence class $[x]$.
  Explicitly, $|C| = \{\bullet,[12],[3],[4]\}$.
  \[\xymatrix@C+2ex@R-2ex{
    M \ar@<1mm>^-{f}[r] \ar@<-1mm>_-{g}[r] & N \ar^-{c}[r] \ar_-{c'}[dr] & C \ar@{..>}[d]^-{k} \\
    && C'
  }\]
  Each flat $G$ of $C$ must be of the form $[F]=\{[x] \mid x \in F\}$ for some flat $F$ of $N$, because
  $G = c(c^{-1}(G)) = [c^{-1}(G)]$. In fact each flat $F$ of $C$ must satisfy $F=[F]$ in $N$.
  Thus the only flats which may be contained in $C$ are
  \[
    \{\bullet\}, \{\bullet,[3]\}, \{\bullet,[4]\}, \{\bullet,[12]\}, \{\bullet,[3],[4]\}, \{\bullet,[12],[3],[4]\}.
  \]
  That leaves only three possibilities for a well-defined geometric lattice $\mathcal{F}$:
  \begin{align}
    &\big\{\{\bullet\},\{\bullet,[12],[3],[4]\}\big\},\label{eq:c1}\\
    &\big\{\{\bullet\},\{\bullet,[12]\},\{\bullet,[3],[4]\},\{\bullet,[12],[3],[4]\}\big\},\text{ or}\label{eq:c2}\\
    &\big\{\{\bullet\},\{\bullet,[12]\},\{\bullet,[3]\},\{\bullet,[4]\},\{\bullet,[12],[3],[4]\}\big\}.\label{eq:c3}
  \end{align}
  Option~\eqref{eq:c1} fails when we set $C'$ to have the same ground set as $C$ with flats~\eqref{eq:c3}, because then the inverse image under $k$ of any rank-1 flat is not a flat.
  Option~\eqref{eq:c2} fails when we give $|C'|=\{[\bullet4],[12],[3]\}$ flats $\big\{\{[\bullet4]\},\{[\bullet4],[12],[3]\}\big\}$, because then the inverse image under $k$ of the rank-0 flat is not a flat.
  Option~\eqref{eq:c3} fails when we endow $|C'|=\{[\bullet34],[12]\}$ with flats $\big\{\{[\bullet34]\},\{[\bullet34],[12]\}\big\}$, because then the inverse image under $k$ of the rank-0 flat is not a flat.
  \qed
\end{proof}

\begin{corollary}\label{cor:monomorphisms}
  A morphism of $\Matr$ is monic if and only if it is injective, and epic if and only if it is surjective.
\end{corollary}
\begin{proof}
  The functor $|-| \colon \Matr \to \cat{FinSet}$ reflects kernel pairs and cokernel pairs. It follows that $|-|$ preserves and reflects monomorphisms and epimorphisms.
  \qed
\end{proof}

\begin{lemma}\label{lem:isomorphisms}
  A morphism of $\Matr$ is an isomorphism if and only if it is bijective and the direct image of any flat is a flat.
\end{lemma}
\begin{proof}
  If $f \colon M \to N$ is an isomorphism, it is bijective by Corollary~\ref{cor:monomorphisms}. 
  The direct image under $f$ of a flat $F$ equals the inverse image under $f^{-1}$ of $F$, and therefore is a flat in $N$ because $f^{-1}$ is strong. The latter argument also establishes the converse.
  \qed
\end{proof}

\section{Adjunctions between subcategories of matroids}\label{sec:adjunctions}

We have seen various classes of matroids: all matroids, simple matroids, free matroids, and loopless matroids. We now study free and cofree constructions translating between these classes. Theorem~\ref{thm:finset} already showed that free matroids over sets exist, and are precisely what we have been calling free matroids. We go on to consider whether the inclusions 
\[
  \FMatr \hookrightarrow \SMatr \hookrightarrow \LMatr \hookrightarrow \Matr
\]
have adjoints. 

\subsection{Pointed matroids}

We first focus on pointed matroids, which play an important role in matroid theory.

\begin{definition}
  Write $\Matrp$ for the category of pointed matroids and strong maps that preserve the distinguished point, and $\FMatrp$, $\SMatrp$, $\LMatrp$ for its full subcategories of pointed free, pointed simple, and pointed loopless matroids.
\end{definition}

\begin{remark}\label{rem:reviewerscomment}
  A loop in a matroid $M$ is a strong morphism $C(1)\to M$ (where 1 denotes the singleton set).
  Hence $\Matrp$ is (isomorphic to) the coslice category $C(1) \slash \Matr$. 
  Since $\Matr$ has coproducts by Proposition~\ref{prop:coproducts}, the forgetful functor $\Matrp\to\Matr$ is monadic.
  Therefore its left adjoint $(-)_{\bullet}$ sends each object $M\in\Matr$ to the coproduct inclusion $C(1)\to M+C(1)$. 
  Moreover, this monad preserves connected colimits. Thus the forgetful functor $\Matrp \to \Matr$ creates limits and connected colimits. In particular it preserves and reflects monomorphisms and epimorphisms. 
  It follows that $\Matrp$ has no products, exponentials or pullbacks; in view of Proposition~\ref{prop:pushouts}, it also follows that $\Matrp$ has all coproducts.
\end{remark}

The proof of Proposition~\ref{prop:equalizers} applies unchanged to the pointed categories to show that equalizers exist. 
As in Remark~\ref{rem:constructlimit}, the category $\FMatrp$ is isomorphic to the category $\cat{FinSet}_\bullet$ of pointed finite sets and pointed functions, and so has all finite (co)limits~\cite[28.9.5]{adamekherrlichstrecker:joyofcats}.

Apart from coproducts, $\Matrp$ does not have many colimits: the proofs of Proposition~\ref{prop:pushouts} and Proposition~\ref{prop:coequalizers} apply unchanged to the corresponding pointed categories to show that there are no coequalizers or pushouts. Because of this, we cannot invoke the adjoint functor theorem. 
We will reason concretely below, to avoid and repair mistakes in the literature.

\begin{lemma}\label{lem:pointed}
  The forgetful functor $\Matrp \to \Matr$ has no right adjoint.
\end{lemma}
\begin{proof}
  Suppose the functor $R$ were a right adjoint. Consider the empty matroid $O$ with $|O|=\emptyset$. For any pointed matroid $N$, there are no functions $N \to O$. But this contradicts the fact that there does exist a pointed strong map $N \to R(O)$, namely the map that sends every element of $N$ to $\bullet$.
  \qed
\end{proof}

The categories $\Matrp$ and $\Matr$ are not equivalent. In particular, $\Matrp$ has a zero object: the one-element matroid is both initial and terminal. This matroid is a terminal object in $\Matr$, but not initial.

\subsection{Free pointed matroids}

We start with right adjoints of functors out of the category of free pointed matroids.

\begin{theorem}\label{thm:rightadjointFMatrLMatr} 
 The category $\FMatrp$ is a coreflective subcategory of $\LMatrp$: the inclusion $\FMatrp \hookrightarrow \LMatrp$ has a right adjoint $F$ defined by
  \[
    |F(M)| = |M|,
    \qquad
    \mathcal{F}_{F(M)} = 2^{|M|} \setminus 2^{|M|\setminus \{\bullet\}},
    \qquad
    F(f) = f.
  \]
\end{theorem}
\begin{proof}
  See~\cite[Theorem~86]{alhawary:thesis} and \cite{alhawarymcrae:axiomatic,alhawary:free}.
  \qed
\end{proof}

The above functor $F$ extends to right adjoints of the inclusions $\FMatrp \hookrightarrow \Matrp$ and $\FMatrp \hookrightarrow \SMatrp$. We now examine whether the functor $F$ itself has a right adjoint for each of those three cases.

\begin{proposition}
  The functors $F\colon \SMatrp\to\FMatrp$ and $F \colon \LMatrp\to\FMatrp$ have no right adjoints.\footnote{The purported right adjoint in~\cite[Theorem~126]{alhawary:thesis} fails for every function mapping $1<n<\card M$ elements to $\bullet$.}
\end{proposition}
\begin{proof}
  Suppose $V$ were a right adjoint. 
  Let $N$ be the matroid with flats $\big\{\{\bullet\},\{\bullet,e\}\big\}$, and let $M$ be the matroid with flats $\big\{\{\bullet\},\{\bullet,a\},\{\bullet,b\},\{\bullet,c\},\{\bullet,a,b,c\}\big\}$.
  There are 8 morphisms $F(M)\to N$.
  Now count the morphisms $M \to V(N)$. For $\SMatrp$:
  \begin{itemize}
    \item If $\card V(N)=2$, it has flats $\big\{\{\bullet\},\{\bullet,e\}\big\}$: there are 5 morphisms;
        \item If $\card V(N)=3$, it has flats $\big\{\{\bullet\},\{\bullet,e_1\},\{\bullet,e_2\},\{\bullet,e_1,e_2\}\big\}$, and 11 morphisms;
        \item If $\card V(N)>3$, there are at least 11 morphisms.
  \end{itemize}
  For $\LMatrp$ we have to consider one extra case: $\F_{V(N)}=\big\{\{\bullet\},\{\bullet,e_1,e_2\}\big\}$, giving 21 morphisms $M \to V(N)$.
  \qed
\end{proof}

The above proof method, of counting morphisms between clever choices of matroids, will be used often and abbreviated in this section.

\begin{theorem}\label{thm:FVHadjunctions}
  The functor $F \colon \Matrp \to \FMatrp$ has adjoints $F \dashv V \dashv H$ given by
  \begin{align*}
    \F_{V(M)} & = \{|M|\}, & V(f)&=f, \\
    H(M) &= F(M'), & H(f)&=f|_{|M'|}.
  \end{align*}
  where $M'$ is got from $M$ by deleting all but loops.
  The functor $H$ has no right adjoint.
  The inclusions $\FMatrp\hookrightarrow\SMatrp$, $\FMatrp \hookrightarrow \LMatrp$, and $\FMatrp \hookrightarrow \Matrp$ have no left adjoint.
\end{theorem}
\begin{proof}
  Identical to that of Theorem~\ref{thm:finset}.
  \qed
\end{proof}

\subsection{Simple pointed matroids}

Next, we turn to the inclusion of simple matroids into larger categories.

\begin{proposition}\label{prop:SMleftadjoint}
  The category \SMatrp~is a reflective subcategory of \Matrp: the inclusion $\SMatrp \hookrightarrow \Matrp$ has a left adjoint $\sip$. The functor $\sip$ has no left adjoint.
\end{proposition}
\begin{proof}
  The first statement follows directly from Theorem~\ref{thm:simplification}.
  For the second, suppose $K \dashv \sip$. 
  Take $M$ to be the pointed matroid with two parallel elements $a,b$ and the loop $\bullet$, and write $F_1$ for its rank-1 flat. Take $S$ to be the pointed simple matroid with elements $e$ and $\bullet$. Let $f$ map $e$ to $F_1$. Its transpose $\hat f$ must map some element $e'$ of $K(S)$ to either $a$ or $b$. But if $e'\mapsto a$ satisfies $f = \epsilon\circ\sip(\hat{f})$ then so does $e'\mapsto b$, and vice versa. Therefore, $\hat f$ cannot be unique.  
  \qed
\end{proof}

\begin{proposition}\label{prop:SMnorightadjoint}
  The inclusion $\SMatrp \hookrightarrow \Matrp$ has no right adjoint.
\end{proposition}
\begin{proof}
  If $R \colon \Matrp \to \SMatrp$ were a right adjoint, there would be a natural isomorphism $F \cong F \circ R \colon \Matrp \to \FMatrp$, whence $R$ must (1) preserve cardinality and (2) reflect surjectivity.
  Now write $S$ for the simple matroid with flats $\big\{\{\bullet\},\{\bullet,1\},\{\bullet,2\},\{\bullet,3\},\{\bullet,1,2,3\}\big\}$. By property (1), $R(S)$ has 4 elements, so its family of flats must be either that of $S$ or
  \[
    \big\{\{\bullet\},\{\bullet,1\},\{\bullet,2\},\{\bullet,3\},\{\bullet,1,2\},\{\bullet,2,3\},\{\bullet,1,3\},\{\bullet,1,2,3\}\big\}.
  \]
  Let $M$ be the matroid with flats $\big\{\{\bullet\},\{\bullet,1\},\{\bullet,2,3\},\{\bullet,1,2,3\}\big\}$. The family of flats of $R(M)$ must be either that of $S$ or
  \[
    \big\{\{\bullet\},\{\bullet,1\},\{\bullet,2\},\{\bullet,3\},\{\bullet,1,2\},\{\bullet,2,3\},\{\bullet,1,3\},\{\bullet,1,2,3\}\big\}.
  \]
  The latter case gives 25 strong maps $S\to M$ but only 14 strong maps $S\to R(M)$.
  In the former case the transpose $\hat{f}$ of any map $f$ must be surjective by property (2), but there are no surjective strong maps $S \to M$.
  \qed
\end{proof}

\subsection{Loopless pointed matroids}

Let us consider adjoints of the inclusion of loopless matroids into larger categories.

\begin{theorem}
  The category $\LMatrp$ is a reflective subcategory of $\Matrp$: the inclusion $\LMatrp \hookrightarrow \Matrp$ has a left adjoint $J$ that deletes every loop except $\bullet$ from objects and acts on morphisms $f \colon M \to N$ as
  \[
    J(f)(e)=\left\{\begin{array}{cc}
    f(e) & \text{ if }f(e)\in J(N) \\
    \bullet & \text{ if }f(e)\notin J(N) \\
    \end{array}\right.
  \]
  The functor $J \colon \Matrp \to \LMatrp$ has no left adjoint.
\end{theorem}
\begin{proof}
  For the unit $\eta_M \colon M\to J(M)$ we may take the strong map that sends every loop to $\bullet$ and every other element to itself. 
  Then a morphisms $f \colon M \to N$ correspond bijectively to $\hat{f}=J(f)$ satisfying $\hat{f} \circ \eta_M = f$.

  Suppose $G \dashv J$. 
  Then $G$ cannot be the constant functor $\bullet$; pick an object $K$ and $\bullet \neq e \in |G(K)|$.
  Let $M$ be the matroid with loops $*$ and $\bullet$. Then strong maps $f \colon K \to J(M)$ must map every element of $K$ to $\bullet$, and any function $G(K)\to M$ is strong. But altering a transpose of $f$ with $e \mapsto *$ into $e \mapsto \bullet$ or vice versa still gives a transpose, so transposes are not unique.
  \qed
\end{proof}

\begin{theorem}\label{thm:LMnorightadjoint}
  The inclusion $\LMatrp\hookrightarrow\Matrp$ has no right adjoint.
\end{theorem}
\begin{proof}
  If $N$ were a right adjoint, 
  there would be a natural isomorphism $F \cong F \circ N \colon \Matrp \to \FMatrp$,
  whence $N$ must (1) preserve cardinality and (2) reflect surjectivity.
  Let $D$ be the matroid with flats $\{\{\bullet\},\{\bullet,c\},\{\bullet,a,b\},\{\bullet,a,b,c\}\}$, and $M$ the matroid with flats $\{\{\bullet,*\},\{\bullet,*,e\}\}$. Now $\card |N(D)|=4$ and $\card |N(M)|=3$ by property (1), leaving two possible choices for $\mathcal{F}_{N(M)}$.
  The first choice is
  $\big\{\{\bullet\},\{\bullet,*\},\{\bullet,e\},\{\bullet,*,e\}\big\}$; then there are 9 strong maps $D\to N(M)$ but 15 strong maps $D\to M$.
  The second choice is $\big\{\{\bullet\},\{\bullet,*,e\}\big\}$; by property (2), epimorphisms $D \to N(M)$ correspond  to epimorphisms $D \to M$, but there are 8 of the former and 4 of the latter.
  Either choice gives a contradiction.
  \qed
\end{proof}

The following theorem summarizes all adjunctions in the pointed case.

\begin{theorem}
  The inclusions have the following adjunctions:
  
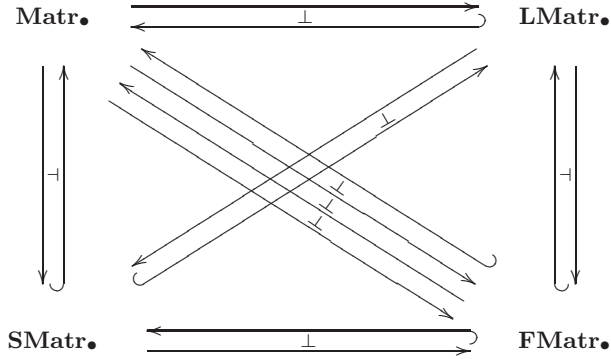
\[\xymatrix@C+25ex@R+15ex{
    *+++++[d]{\Matrp} 
    \ar@{}|-{\perp}[r]
    \ar@<-1ex>@{<-^{)}}[r]
    \ar@<-1ex>[d]
    \ar[dr]
    \ar@<-1ex>@{}|(.55){\lldash}[dr]
    \ar@<-4ex>[dr]
    \ar@<-3ex>@{}|(.55){\lldash}[dr]
    & 
    *+++++[d]{\LMatrp}
    \ar@<-1ex>@{<-}[l]
    \ar@<1ex>[d]
    \ar@<-2ex>[dl]
    \\
    *++++++[d]{\SMatrp}
    \ar@{}|-{\perp}[r]
    \ar@<-1ex>[r]
    \ar@<-1ex>@{^{(}->}[u]
    \ar@{}|-{\dashv}[u]
    \ar@{^{(}->}[ur]
    \ar@<1ex>@{}|(.66){\rrdash}[ur]
    &
    *++++++[d]{\FMatrp}
    \ar@<-1ex>@{^{(}->}[l]
    \ar@{}|-{\dashv}[u]
    \ar@<1ex>@{_{(}->}[u]
    \ar@<-2ex>@{_{(}->}[ul]
    \ar@<-1ex>@{}|(.45){\lldash}[ul]
    \ar@<2ex>[ul]
  }\]
  The functors in the above diagram have no adjoints other than those indicated.
\end{theorem}
\begin{proof}
  Collate the previous results in this section.
  \qed  
\end{proof}

\subsection{Unpointed categories}

We would like to have a translation principle between pointed and unpointed versions of our matroid categories. Unfortunately, as we have seen in the previous Section, the categories $\Matrp$ and $\Matr$ are not equivalent. So our results for the pointed categories do not necessarily translate into the unpointed versions, and we have to reason directly for the (non)existence of adjoints.

\begin{proposition}\label{freefun} 
  The category $\FMatr$ is a coreflective subcategory of $\SMatr$: the inclusion $\FMatr \hookrightarrow \SMatr$ has a right adjoint $F$ defined by:
  \[
    |F(M)|=|M|,
    \qquad \mathcal{F}_{F(M)} = 2^{|M|},
    \qquad F(f)=f.
  \]
  It extends to right adjoints of the inclusions $\FMatr \hookrightarrow \Matr$ and $\FMatr \hookrightarrow \LMatr$.
\end{proposition}
\begin{proof}
  Simply take the unit $\eta_M \colon M \to F(M)$ to be the identity on elements, and $\hat{f}=f$.
  \qed
\end{proof}

The functor $F \colon \Matr\to\FMatr$ has a right adjoint $V \colon \FMatr\to\Matr$, which in turn has a right adjoint $H \colon \Matr\to\FMatr$, both defined as in Theorem~\ref{thm:FVHadjunctions}, and $H$ has no right adjoint; we omit the proofs.

\begin{proposition}
  The functor $F \colon \SMatr\to\FMatr$ has no right adjoint.
\end{proposition}
\begin{proof}
  Suppose $V$ were a right adjoint. 
  Let $M$ be the free matroid on $m$ elements, and let $D$ be the matroid with flats $\big\{\emptyset,\{a\},\{b\},\{c\},\{a,b,c\}\big\}$. There are $m^3$ strong maps $F(D)\to M$.
  If $V(M)$ contains $D$ as a submatroid, then there are at least 9 maps $D\to V(M)$, because there are 9 maps $D\to D$. Otherwise, for the map $f$ to be strong, every element of $D$ must map to the same element of $V(M)$, giving $\card |V(M)|$ many maps $D\to V(M)$, so $\card |V(M)| = m^3$.
  Now take $D$ to be the simple matroid with $|M|=\{e\}$. There are $m$ maps $F(D)\to M$, but $m^3$ maps $D\to V(M)$, and in general $m\ne m^3$.
  \qed
\end{proof}

\begin{proposition}
  The functor $F \colon \LMatr\to\FMatr$ has a right adjoint $U$ given by:
  \[
    |U(M)|=|M|,
    \qquad 
    \mathcal{F}_{U(M)}=\{\emptyset,|M|\},
    \qquad
    U(f)=f.
  \]
  The functor $U$ has no right adjoint.
\end{proposition}
\begin{proof}
  Take the unit to be the identity on elements and $\hat{f}=f$ to establish $F \dashv U$.
  Suppose $U \dashv G$. 
  Take $D$ to be the free matroid on 2 elements, so $U(D)$ is the matroid with 2 parallel elements. 
  Taking $M=D$ gives exactly 2 morphisms $U(D)\to M$, whereas there are $\card G(M)^2$ morphisms $D\to G(M)$. 
  \qed
\end{proof}

\begin{proposition}
  None of the inclusions $\FMatr \hookrightarrow\SMatr$, $\FMatr\hookrightarrow\LMatr$, and $\FMatr\hookrightarrow\Matr$ have a left adjoint.
\end{proposition}
\begin{proof}
  We treat the first inclusion as the other two are similar. Suppose $K$ were a left adjoint.
  Let $N$ be the free matroid on 2 elements, and let $M$ be the matroid with flats
  $\big\{\emptyset,\{a\},\{b\},\{c\},\{d\},\{a,b\},\{a,c\},\{a,d\},\{b,c\},\{b,d\},\{c,d\},\{a,b,c,d\}\big\}$.
  There are 8 strong maps $M\to N$, and $2^{\card K(M)}$ strong maps $K(M)\to N$, so $\card K(M)=3$.

  As in Theorem~\ref{thm:finset}, the unit $\eta$ has to map $a \neq b \in |M|$ to the same element of $K(M)$ and each of the other two elements of $M$ to the two remaining elements of $K(M)$. Picking $N$ with 3 elements gives a contradiction as in that Theorem.
  \qed
\end{proof}

The inclusions $\SMatr \hookrightarrow \Matr$ and $\LMatr \hookrightarrow \Matr$ have no left adjoint $K$: if $M$ has at least one loop then there are no morphisms $M \to K(M)$ at all.

\begin{corollary}
  The category \SMatr~is a reflective subcategory of \LMatr: the inclusion has a left adjoint $\si$ as in Lemma~\ref{thm:simplification}. The functor $\si$ has no left adjoint.
\end{corollary}
\begin{proof}
  The first statement follows from Remark~\ref{prop:L} as before.
  The second statement is proven just like in Proposition~\ref{prop:SMleftadjoint}.
  \qed
\end{proof}

\begin{proposition}
  The inclusions $\SMatr \hookrightarrow\Matr$ and $\SMatr \hookrightarrow \LMatr$ have no right adjoint.
\end{proposition}
\begin{proof}
  If $R \colon \Matr \to \SMatr$ were a right adjoint, it would satisfy properties (1) and (2) as in Proposition~\ref{prop:SMnorightadjoint}.
  Let $S$ be the matroid with flats $\big\{\emptyset,\{1\},\{2\},\{3\},\{1,2,3\}\big\}$, and let $M$ be the matroid with flats $\big\{\emptyset,\{1\},\{2,3\},\{1,2,3\}\big\}$. The family of flats of $R(M)$ must be either that of $S$ or
  \[
    \big\{\emptyset,\{1\},\{2\},\{3\},\{1,2\},\{2,3\},\{1,3\},\{1,2,3\}\big\}.
  \]
  In either case, the number of strong maps $S\to R(M)$ does not equal the number of strong maps $S\to M$. 
  \qed
\end{proof}

\begin{theorem}
  The inclusion $\LMatr \hookrightarrow\Matr$ has no right adjoint.
\end{theorem}
\begin{proof}
  If $N$ were a right adjoint, it would satisfy properties (1) and (2) as in Theorem~\ref{thm:LMnorightadjoint}.
  Let $D$ be the matroid with flats $\{\emptyset,\{c\},\{a,b\},\{a,b,c\}\}$, and $M$ the matroid with flats $\{\{a,b\},\{a,b,c\}\}$. Now $\card |N(D)| = \card |N(M)|=3$ by property (1), giving four choices for $\mathcal{F}_{N(M)}$.
  First, $\big\{\emptyset,\{a,b,c\}\big\}$ gives 27 strong maps $D\to N(M)$.
  Second, $\big\{\emptyset,\{c\},\{a,b\},\{a,b,c\}\big\}$ gives 15 strong maps $D\to N(M)$.
  Third, $\big\{\emptyset,\{a\},\{b\},\{c\},\{a,b,c\}\big\}$ gives 9 strong maps $D\to N(M)$.
  Fourth and last, $\big\{\emptyset,\{a\},\{b\},\{c\},\{a,b\},\{b,c\},\{a,c\},\{a,b,c\}\big\}$ gives 9 strong maps $D\to N(M)$.
  As there are 15 strong maps $D\to M$, it must be the second option.
  But if $D$ has flats $\big\{\emptyset,\{a\},\{b\},\{c\},\{a,b,c\}\big\}$, there are 9 strong maps $D\to N(M)$ and 15 strong maps $D\to M$.
  \qed
\end{proof}

The following theorem summarizes all adjunctions in the unpointed case.

\begin{theorem}
  The inclusions have the following adjunctions:
  
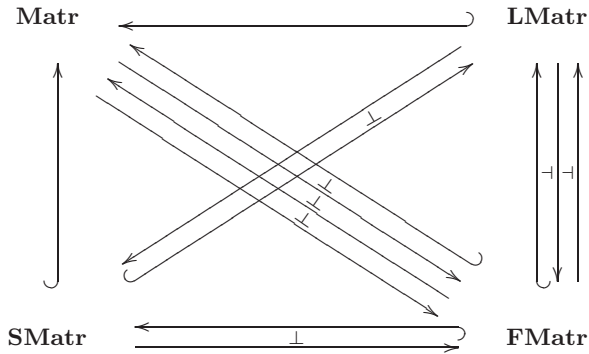
\[\xymatrix@C+25ex@R+15ex{
    *+++++[d]{\Matr} 
    \ar@<-1ex>@{<-^{)}}[r]
    \ar[dr]
    \ar@<-1ex>@{}|(.55){\lldash}[dr]
    \ar@<-4ex>[dr]
    \ar@<-3ex>@{}|(.55){\lldash}[dr]
    & 
    *+++++[d]{\LMatr}
    \ar@<1ex>[d]
    \ar@<-2ex>[dl]
    \\
    *++++++[d]{\SMatr}
    \ar@{}|-{\perp}[r]
    \ar@<-1ex>[r]
    \ar@<-1ex>@{^{(}->}[u]
    \ar@{^{(}->}[ur]
    \ar@<1ex>@{}|(.66){\rrdash}[ur]
    &
    *++++++[d]{\FMatr}
    \ar@<-1ex>@{^{(}->}[l]
    \ar@{}|-{\dashv}[u]
    \ar@<1ex>@{_{(}->}[u]
    \ar@<-2ex>@{_{(}->}[ul]
    \ar@<-1ex>@{}|(.45){\lldash}[ul]
    \ar@<-2ex>@{}|-{\dashv}[u]
    \ar@<-3ex>[u]
    \ar@<2ex>[ul]
  }\]
  The functors in the above diagram have no adjoints other than those indicated.
\end{theorem}
\begin{proof}
  Collate the previous results in this subsection.      
  \qed
\end{proof}

\section{Deletion and contraction}\label{sec:deletioncontraction}

Let us recall some standard terminology from matroid theory.

\begin{definition}\label{def:deletion}
  Let $M$ be a matroid. The \emph{deletion} of $Y\subseteq |M|$ from $M$ is the matroid $M \setminus Y$ with ground set $|M| \setminus Y$ and rank function $X \mapsto \rank_M(X)$. The resulting matroid is said to be \emph{embedded} in $M$, and the strong map $M \setminus Y \to M$ is called an \emph{embedding}.
  The \emph{contraction} of $M$ by $Z\subseteq |M|$ is the matroid $M/Z$ with ground set $|M| \setminus Z$ and rank function $X \mapsto \rank_M(X \cup Z) - \rank_M(Z)$. 
  A \emph{minor} of $M$ is the matroid resulting from a sequence of deletions and contractions of $M$.   
\end{definition}

We can identify the (categorical) subobjects, that is, equivalence classes of monomorphisms $M \rightarrowtail N$, where two such monomorphisms are equivalent when there is an isomorphism $M \simeq M'$ making the triangle commute. In terms of their domains, subobjects of $N$ correspond to the matroids from which there exists an injective strong map into $N$.\footnote{Some publications~\cite{kahn1982varieties,kung2001twelve,white:matroids} incorrectly state that subobjects in this category coincide with matroid minors. For a counterexample to this statement, the canonical map $F(M) \to M$ is injective, but $F(M)$ is not generally a minor of $M$.}

Next we move to the dual notion of subobjects, (categorical) quotients: that is, equivalence classes of epimorphisms $M \twoheadrightarrow N$ where two such epimorphisms are equivalent when there is an isomorphism $N \simeq N'$ making the triangle commute.

\begin{definition}
  A matroid $Q$ is a \emph{(matroid) quotient} of $M$ if there exist a matroid $N$ and some $X\subseteq |N|$ so that $M=N\setminus X$ and $Q=N/X$. 
\end{definition}

Hence by definition, (matroid) quotients are strong maps that are composed of a contraction after an embedding. The rest of this section proves that quotients are precisely the bijective strong maps, from which it follows by Corollary~\ref{cor:monomorphisms} that matroid quotients are not categorical quotients in the category of matroids. This also leads us to a characterisation of subobjects; these are embeddings followed by matroid quotient maps.

Theorem~\ref{thm:quotients} in a later section shows that matroid quotients do correspond to categorical quotients in a related category, that we now introduce.

We can derive that contractions, like embeddings, are strong maps.

\begin{corollary}\label{cor:contractions}
  If $M$ is a pointed matroid and $Z \subseteq |M|$, there is a strong map $M \to M/Z$.
\end{corollary}
\begin{proof}
  For $A=X\cup Z$, it follows from Definition~\ref{def:matroid} that
  \[
    \rank_M(A)+\rank_M(Y) \geq \rank_M(A\cup Y)+\rank_M(A\cap Y),
  \]
  so $\rank_M(Y)-\rank_M(A\cap Y) \geq \rank_M(A\cup Y)-\rank_M(A)$, whence
  \begin{align*}
    \rank_M(Y)-\rank_M(X)&\geq \rank_M(Y\cup Z)-\rank_M(X\cup Z)\\
    &=\rank_M(Y\cup Z)-\rank_M(Z)-(\rank_M(X\cup Z)-\rank_M(Z))\\
    &=\rank_{M/Z}(Y)-\rank_{M/Z}(X).
  \end{align*}
  Lemma~\ref{lem:latticemaps} now establishes the result.
  \qed
\end{proof}

By the standard definition of the contraction operation, the strong map corresponding to contraction acts as the identity on noncontracted elements and maps the rest to the distinguished loop. Alternatively, one may redefine the contracted matroid on the original ground set, keeping the original elements as loops. In the latter case, the contraction map acts as the identity on all elements.

Finally we establish that matroid quotients are precisely bijective strong maps.

\begin{lemma}\label{lem:quotients}
  A function $f \colon M \to N$ is a bijective strong map if and only if it factors as an embedding $M \to M\backslash X=Q$ followed by a contraction $Q=N/X \to N$.
  \[\xymatrix@R-1ex{
    & Q \ar^-{\text{contraction}}[dr] \\
    M \ar_-{f}[rr] \ar^-{\text{embedding}}[ur] && N
  }\]
\end{lemma}
\begin{proof}
  Sufficiency follows from Corollary~\ref{cor:contractions}, necessity is proven by Higgs~\cite{higgs1968strong}.
  \qed
\end{proof}

\begin{lemma}
  Every contraction is a coequalizer in $\Matrp$, but not conversely.
\end{lemma}
\begin{proof}
  Suppose $c \colon N \to N \slash Z$ is a contraction with $c(z)=\bullet$ for $Z \subseteq |N|$.
  Let $M$ be the free matroid on $|M|=|N|$.
  Define $f,g \colon M \to N$ by $f(x)=x$, and 
  \[
        g(x)=\left\{\begin{array}{cc} 
          x, & \mbox{ if }x\notin Z, \\ 
          \bullet, & \mbox{ if }x\in Z. \\
        \end{array}\right.    
  \]
  Then $c$ is a coequalizer of $f$ and $g$.

  Conversely, keeping $f$ the same but letting $g$ send all nonloop elements to the same nonloop element  results in a coequalizer that is not a contraction.
  \qed
\end{proof}

\section{Factorization}\label{sec:factorization}

In this section we study how morphisms between matroids can be factored into easier classes of strong maps.
Let us first recall the basic definition~\cite{adamekherrlichstrecker:joyofcats}.

\begin{definition}\label{def:factorization}
  A \emph{weak factorization system} in a category consists of two classes of morphisms $\mathcal{L}$ and $\mathcal{R}$ such that:
  \begin{itemize}
  \item every morphism $f$ factors as $f = r \circ l$ for some $l \in \mathcal{L}$ and $r \in \mathcal{R}$;
  \item both $\mathcal{L}$ and $\mathcal{R}$ contain all isomorphisms;
  \item if $l,l' \in \mathcal{L}$, $r,r' \in \mathcal{R}$, and arbitrary morphism $f,g$ make the following diagram commute, then there is a \emph{fill-in} $h$ making both squares commute:
  \begin{equation}\label{eq:fillin}\begin{aligned}\xymatrix{
    M \ar_-{f}[d] \ar^-{l}[r] & I \ar@{-->}^-{h}[d]\ar[r]^r & N\ar[d]^g \\
    M' \ar_-{l'}[r] & I'\ar_-{r'}[r]& N'
  }\end{aligned}\end{equation}
  \end{itemize}
  In an \emph{orthogonal factorization system} the fill-in $h$ is additionally unique.
\end{definition}

The standard example of an orthogonal factorization system is that every function between sets factors as an epimorphism followed by a monomorphism. The category of matroids has a very similar orthogonal factorization system.

\begin{lemma}\label{lem:factorization1}
  The category $\Matrp$ has an orthogonal factorization system where $\mathcal{L}$ consists of epimorphisms and $\mathcal{R}$ consists of embeddings.
\end{lemma}
\begin{proof}
  The fill-in is the restriction of $g$ to the image of $l$, which is a strong map.
  \qed  
\end{proof}

Epimorphisms in $\Matrp$ can be further decomposed into a quotient followed by a lattice-preserving map~\cite[page~228]{white:matroids}. This would give another orthogonal factorization system with $\mathcal{L}$ quotients and and $\mathcal{R}$ lattice-preserving maps followed by embeddings, except that the fill-in $h$, which has to be the function $l' \circ f \circ l^{-1}$ by Lemma~\ref{lem:quotients}, need not be a strong map.

Finally, any quotient can be decomposed into an embedding followed by a contraction~\cite[page~228]{white:matroids}. Again this does not give a weak factorization system, but there exists a construction for the minimal matroid through which a quotient factors like this~\cite{kennedy1975majors}.

\begin{theorem}\label{thm:factorization2}
  The category $\Matrp$ has an orthogonal factorization system where $\mathcal{L}$ consists of lattice-preserving maps and $\mathcal{R}$ consists of maps that are injective on elements of each rank-1 flat.
\end{theorem}
\begin{proof}
For simplicity, and without loss of generality, we identify isomorphic lattices in this proof.
  Given a morphism $f \colon M \to N$, a matroid $I$ is uniquely specified by saying that $L(I)=L(M)$, that $I$ has the same rank-0 flat as $N$, and that the nonloop elements of each rank-1 flat $F_i$ are copies of the elements of the elements of $f(F_i)$ (including extra copies of the loops).
  Now $f$ decomposes as $\smash{M \stackrel{l}{\to} I \stackrel{r}{\to} N}$, where $l$ is the lattice-preserving map that acts as $f$ on elements when ignoring indices $i$, and $r$ is the map with $L(r)=L(f)$ that sends each element of $I$ to the unindexed version of the element in $N$.
  Then $l$ is strong because $L(l)=\id{L(M)}$, and $r$ is strong because $L(r)=L(f)$. 
  By definition $l$ is lattice-preserving, and $r$ is obviously injective on elements of each rank-1 flat.
  Both $\mathcal{L}$ and $\mathcal{R}$ contain all isomorphisms by Lemma~\ref{lem:isomorphisms}.

  For the fill-in $h$ in~\eqref{eq:fillin}, take the morphism with $L(h)=L(f)$ that acts on $g$ on elements when ignoring indices. 
  This is by construction the unique strong map that makes both squares commute, as we conclude by considering what $h$ does on lattices and on elements.
  \qed
\end{proof}

\begin{remark}\label{remark:quillen}
  The category $\Matrp$ has a double factorization system~\cite{pultrtholen:quillen}: every morphism decomposes as a lattice-preserving map followed by an epimorphism injective on elements of each rank-1 flat followed by an embedding.
  This would give an \emph{orthogonal Quillen factorization system}, whose \emph{fibrations} are the maps that distinguish between parallel elements (\ie are injective on elements of each rank-1 flat), whose \emph{cofibrations} are the epimorphisms, and whose \emph{weak equivalences} are the monomorphisms $f$ such that $L(f)$ preserves rank, except that $g$ need not be a weak equivalence when $f$ and $g \circ f$ are.
\end{remark}
\begin{proof}
  As embeddings are by construction injective, the orthogonal factorization systems of Lemma~\ref{lem:factorization1} and Theorem~\ref{thm:factorization2} are comparable in the sense of~\cite[Proposition~2.6]{pultrtholen:quillen}, and hence form a double factorization system~\cite[Theorem~2.7]{pultrtholen:quillen}. The middle class of morphisms are those that are both epic and injective on rank-1 flats, that is, strong maps that are bijective when restricted to rank-1 flats.
  As for the purported orthogonal Quillen factorization system, the weak equivalences are the strong maps that factor as a lattice-preserving map followed by an embedding, which are precisely the strong maps as in the statement, but~\cite[Theorem 3.10.1b]{pultrtholen:quillen} is not satisfied.
  For an explicit counterexample, take $\F(M_1)=\{\{\bullet\}\},\F(M_2)=\{\{\bullet\},\{\bullet,*\}\},\F(M_3)=\{\{\bullet,*\}\}$ and $f:M_1\to M_2$, $g:M_2\to M_3$ mapping each element to itself; $f$ and $g\circ f$ are clearly embeddings whereas $g$ cannot be expressed as a lattice-preserving map followed by an embedding.  
  \qed
\end{proof}

\section{Functors}\label{sec:functors}

This section considers functors into and out of the category of matroids from the categories of geometric lattices, vector spaces, and graphs. We will need the following general notions.

\begin{definition}\label{def:nearlyfull}
  A functor $F \colon \cat{C}\to\cat{D}$ is \emph{nearly full} when any morphism $g$ in $\cat{D}$ between objects in the image of $F$ is of the form $g=F(f)$ for a morphism $f$ in $\cat{C}$.
\end{definition}

A full functor is nearly full, but the converse is not true in general: the functions $\cat{C}(A,B)\to\cat{D}(FA,FB)$ need not be  surjective for all objects $A$ and $B$ in $\cat{C}$. 

For the next notion, recall that the category of left actions on sets of a monoid $M$ is monoidal~\cite[VII.4]{maclane:categories}, so one can consider categories enriched in it. Any locally small monoidal category is an example via scalar multiplication~\cite{heunen:semimoduleenrichment,abramsky2005abstract}.

\begin{definition}\label{def:nearlyfaithful}
  Let $M$ be a monoid, and $\cat{C}$ a category enriched in left $M$-actions.
  A functor $F \colon \cat{C} \to \cat{D}$ is \emph{nearly faithful} when $F(f)=F(g)$ implies $m_1\cdot f=m_2\cdot g$ for some $m_1,m_2 \in M$.
\end{definition}

Any faithful functor is nearly faithful, but the converse it not true in general. Intuitively, a nearly faithful functor is `faithful up to a scalar in $M$'.

\subsection{Geometric lattices}\label{sec:functors:lattices}

We start with a functor from the category of matroids to the following category of geometric lattices, extending Example~\ref{ex:lattice}.

\begin{definition}
  Write $\GLat$ for the category with geometric lattices as objects and as morphisms functions that preserve joins and that map every atom to either an atom or to the least element. Isomorphic lattices are identified.
\end{definition}

Note that morphisms in $\GLat$ are completely determined by their action on atoms, as in geometric lattices all elements are joins of atoms.

\begin{proposition}\label{prop:L}
  There is a functor $L \colon \Matr \to \GLat$ that sends a matroid $M$ to its lattice $L(M)$ of flats ordered by inclusion, and sends a strong map $f \colon M \to N$ to the function $L(f)\colon L(M)\to L(N)$ given by $X \mapsto \clos(f(X))$.
\end{proposition}
\begin{proof}
  Lemma \ref{lem:latticemaps} guarantees that $L(f)$ is a well-defined morphism in $\GLat$. 
  If $X \in L(M)$ then $X=\clos(X)$, so $L$ preserves identities. 
  Finally, to see that $L$ preserves composition, let $f \colon M \to N$ and $g \colon N \to P$ be strong maps.
  If $X \in L(M)$ then $\clos(f(M)) \in L(N)$, so, because $g$ is strong, $\clos(g(f(X)))=\clos(g(\clos(f(X))))$. So $L(g \circ f) = L(g)\circ L(f)$.
  \qed
\end{proof}

The functor $L$ is surjective on objects, and injective on the objects of $\SMatr$, but not injective on objects in general.

\begin{lemma}\label{lem:Lnotfaithful}
  The functor $L$ is not faithful.
\end{lemma}
\begin{proof}
  Write $M$ for the matroid on the ground set $\{0,1,2\}$ with flats $\big\{\{0\},\{0,1,2\}\big\}$. 
  The function $f \colon M \to M$ given by $f(0)=0$, $f(1)=2$, and $f(2)=1$ is a strong map.
  Now $L(f)=L(\id[M])$ both equal the identity function on $\F(M)$.
  \qed
\end{proof}

The functor $L$ is not full. To see this, consider the unique matroid on ground set $\emptyset$. Any matroid $M$ allows a unique morphism $L(M) \to L(\emptyset)$, but clearly there is no strong map $M \to \emptyset$ for nonempty $M$. This is the only obstruction to fullness in the following sense.

\begin{proposition}\label{prop:Lfull}
  The functor $L$ is nearly full; its analog defined on $\Matrp$ is full.
\end{proposition}
\begin{proof}
  Let $M$ and $N$ be matroids, and let $g \colon L(M)\to L(N)$ be a morphism in $\GLat$. 
  First assume that $M$ and $N$ are pointed.
  Construct a function $f \colon M \to N$ as follows.
  For every rank-1 flat $X$ in $M$ for which $g(X)$ is the least element of $L(N)$, define $f(x)=\bullet$ for all $x \in X$.
  For every rank-1 flat $X$ in $M$ for which the flat $g(X)$ in $N$ has rank 1, let $f$ map all $x \in X$ that have not yet been accounted for to an element of $g(X)$ which is not a loop.
  Lemma~\ref{lem:latticemaps} shows that $f$ is a strong map, and by construction $L(f)=g$.
  Hence $L \colon \Matrp \to \GLat$ is full.

  Finally, observe that in general $L(M) \simeq L(M\p)$.
  It follows from the argument above that $L \colon \Matr \to \GLat$ is nearly full.
  \qed
\end{proof}

We will call a morphism $f$ is \emph{lattice-preserving} when $L(f) = \id{}$.
As promised in Section~\ref{sec:category}, we can now prove that matroid quotients are precisely categorical quotients in $\GLat$ via the functor $L$.

\begin{lemma}\label{lem:epimorphisms:lattices}
  Epimorphisms in $\cat{GLat}$ are precisely the surjective morphisms.
\end{lemma}
\begin{proof}
  That surjective maps are epimorphic is clear. The other direction follows by combining Lemma~\ref{lem:isomorphisms}, Corollary~\ref{cor:monomorphisms}, and Proposition~\ref{prop:Lfull}.
  \qed
\end{proof}

\begin{theorem}\label{thm:quotients}
  The functor $L \colon \Matrp \to \GLat$ maps matroid quotients to categorical quotients.
  The restricted functor from the subcategory of pointed matroid quotients to the subcategory of categorical quotients is not full but nearly full.
\end{theorem}
\begin{proof}
  By Lemma~\ref{lem:quotients}, matroid quotients $f$ are surjective functions. Hence $L(f)$ must also be surjective, which is a categorical quotient by Lemma~\ref{lem:epimorphisms:lattices}.

  Not every surjective morphism $g \colon L(M)\to L(N)$ in $\GLat$ has a matroid quotient $f$ with $L(f)=g$:
  if $N$ is the matroid with $1+\card |M|$ loops and no other elements, then all morphisms into $L(N)$ are surjective, but there are no matroid quotients $M\to N$. Hence $L$ is not full.

  However, given a surjective map $g \colon L(M) \to L(N)$ in $\GLat$, take $M'$ to have $L(M')=L(M)$ with ground set the elements of $L(M)$ with $r(x)\le1$ (where the bottom element is a loop), and define $N'$ by $L(N')=L(N)$ and by setting $F \subseteq |N'|$ to be a flat in $N'$ when $F=g^{-1}(F)\cap\{x:x\in L(M),r(x)\le1\}$. Then there is clearly a matroid quotient $f \colon M' \to N'$ with $L(f)=g$, making $L$ nearly full.
  \qed
\end{proof}

There is also a functor in the opposite direction.

\begin{proposition}\label{prop:S}
  There is an embedding $S\colon \GLat\to\Matrp$ that:
  \begin{itemize}
  \item acts on objects $G$ by letting $|S(G)|$ be the set of atoms of $G$ together with a loop $\bullet$ (so that the flats of $S(G)$ form a lattice isomorphic to $G$);
  \item acts on morphisms $f \colon G \to H$ as follows: if $f$ sends an atom of $G$ to the least element of $H$, then $S(f)$ sends the corresponding element of $S(G)$ to $\bullet$; if $f$ sends an atom $A$ of $G$ to an atom $B$ of $H$, then $S(f)$ also sends $A$ to $B$.
  \end{itemize}
\end{proposition}
\begin{proof}
  Proposition~\ref{prop:Lfull} guarantees that $S(f)$ is a legitimate morphism in $\Matrp$.
  Clearly $S$ preserves identities. Seeing that it also preserves composition comes down to noticing that for simple matroids $M$, a strong map $M \to N$ is completely defined by its action on $L(M)$, and that adding a loop does not change this.

  Faithfulness is a direct consequence of the fact that there is a one-to-one correspondence between atoms of the lattice and nonloop elements of the matroid's ground set; different mappings between atoms therefore give rise to differents functions between ground sets. Fullness follows from Lemma~\ref{lem:latticemaps}. Finally, $S$ is injective on objects because there is only one lattice of flats for every matroid.
  \qed
\end{proof}

\begin{theorem}\label{thm:LSreflection}
  There is a reflection $L \dashv S$.
\end{theorem}
\begin{proof}
  When the codomain of a strong map $f$ is a simple matroid with a single added loop and the action of $f$ on rank-1 flats is known, then $f$ is known; there is only one element in each rank-1 flat and the rank-0 flat that each element of the domain can map to. Hence $\Matrp(M,N)\simeq\GLat(L(M),L(N))$ for pointed matroids $M$ and $N$ if $N$ is simple. Moreover, if $N$ is simple then it is in the image of $S$ (up to isomorphism), and by the definition of the action of $S$ on objects we have $G \simeq L(S(G))$ for all geometric lattices $G$. Therefore $\Matrp(M,S(G))\simeq\GLat(L(M),N)$, and this bijection is natural.
  Finally, note that the counit $L(S(G)) \to G$ is an isomorphism, making the adjunction into a reflection.
  \qed
\end{proof}

Any adjunction gives rise to a monad, and in this case we recover the following standard matroid operation~\cite[Section~1.7]{oxley:matroids}.

\begin{definition}\label{def:simplification}
  A \emph{simplification} $\si(M)$ of a matroid $M$ is a matroid obtained by deleting all the loops and all but one element in each rank-1 flat.
\end{definition}

\begin{theorem}\label{thm:simplification}
  Simplification is a monad $\sip= S \circ L \colon \Matrp \to \Matrp$.
  Its category of Eilenberg-Moore algebras is (isomorphic to) $\SMatrp$.
\end{theorem}
\begin{proof}
  By definition $S(L(M))$ is a simplification of $M$.
  Because the monad unit $M \to S(L(M))$ sends each element to its closure, the unit law for Eilenberg-Moore algebras $\si(M) \to M$ implies that $\clos(x)=\{x,\bullet\}$ for nonloop $x\in M$; that is, $M$ is simple, and the structure map has to be the map sending $\{x,\bullet\}\mapsto x$.
  Hence morphisms of Eilenberg-Moore algebras just come down to strong maps.
  \qed
\end{proof}

We now briefly turn our attention to factorization systems, which we discussed in the previous section. Clearly factorization systems in $\Matrp$ are strongly linked to decomposing morphisms in $\GLat$.
Call a morphism $f \colon A \to B$ in $\GLat$ a \emph{contraction} when its restriction to the interval $[\bigvee f^{-1}(0_B),1_A]$ is an isomorphism; call $g \colon A\to B$ an \emph{embedding} when its corestriction to the interval $[0_B,g(1_A)]$ is an isomorphism. This is in line with the definitions of contraction and embedding for matroids: $L(f)$ is a contraction/embedding if $f$ is a contraction/embedding, and the converse holds for simple matroids. A matroid $M$ is a minor of a matroid $N$ only if $L(M)$ is a subobject of $L(N)$, and the converse holds for simple matroids.

\begin{lemma}
  The category $\GLat$ has a weak factorization system where $\mathcal{L}$ consists of embeddings and $\mathcal{R}$ consists of contractions; it is not an orthogonal factorization system.
\end{lemma}
\begin{proof}
  Every surjective map $f \colon M\to N$ in $\Matrp$ factorizes as $f=\tau \circ r\circ l$, where $l$ is an embedding, $r$ is a contraction and $\tau$ is a lattice-preserving map. Then $L(f)=L(\tau\circ r\circ l)=L(\tau)\circ L(r)\circ L(l)=L(r)\circ L(l)$ with $L(l) \in \mathcal{L}$ and $L(r) \in \mathcal{R}$. 
  The middle object $I$ is the lattice of the $n$th Higgs lift of $N$ towards $M+F_n$ along $f_+$, where $n$ is the \emph{nullity} $\rank_M(1_M) - \rank_N(f(1_M))$ of $f$, where $F_n$ is the free matroid on $n$ elements, and where $L(f_+)$ coincides with $L(f)$ on the atoms of $L(M)$ and sends atoms of $L(F_n)$ to the bottom of $L(N)$. 
  Explicitly, $L(I)$ is the sublattice $\{X \mid X=\bigvee f^{-1}(f(X))\}\cup\{X \mid \rank(f(X))=\rank(X)\}$ of $L(M)$, the canonical embedding $L(r) \colon L(M) \to L(I)$ corestricts to the identity, and the canonical contraction $L(l) \colon L(I)\to L(N)$ acts as $L(f_+)$. 
  All of this extends to maps that are not surjective on flats~\cite{higgs1968strong}.
  For the fill-in~\eqref{eq:fillin}, define $h \colon I \to I'$ as the morphism that sends every atom $a$ of $I$ that is also in $L(M)$ to $L(f(a))$, and all other atoms to the bottom of $I'$.
  \qed
\end{proof}

The rest of this subsection shows that if $\GLat$ has an orthogonal factorization system, then it must induce a factorization system in $\Matrp$.

\begin{proposition}
  For every $N \in \Matrp$, the functor $L^N \colon \Matrp/N \to \GLat/L(N)$ between slice categories has a right adjoint $R^N$ that is full.
\end{proposition}
\begin{proof}
  For an object $l \colon G \to L(N)$ in $\GLat/L(N)$, where $G$ has atoms $a_i$, define $R^N_l(G)$ to be the matroid with flats $G$ and the same rank-0 flat as $N$ where the nonloop elements of each rank-1 flat $F_i$ with $L(F_i)=a_i$ are copies, indexed by $i$, of the elements of $L^{-1}(l(a_i))$, including extra copies of the loops. Then $R^N(l) \colon R^N_l(G)\to N$ is the map with $L(R^N(l))=l$ that sends each element in $R^N_l(G)$ to the unindexed version of the element in $N$.

  For a morphism $f \colon l_1 \to l_2$ in $\GLat/L(N)$, define $R^N(f) \colon R^N_{l_1}(G)\to R^N_{l_2}(G)$ as the map with $L(R^N(f))=f$ that acts as the identity on elements. Since $R^N(l_1)\circ R^N(f)=R^N(l_2)$, this is a legitimate morphism in $\Matrp/N$, and it clearly respects identities and composition, giving a well-defined functor $R^N$.

  Moreover, we claim that the there uniquely exist morphisms $\eta_f$ and $\hat{h}$ making the following diagrams commute for every $f \colon M \to N$ in $\Matrp$,  every $g \colon K \to L(M)$ in $\GLat$, and every strong map $h \colon M \to \mathrm{dom}(R^N(g))$ with $R^N(g) \circ h = f$:
  \[\xymatrix@C+3ex@R-3ex{
     M\ar[rr]^{\eta}\ar[dr]^f\ar[dddr]^h&&M''\ar[dl]_{R^N(L(f))}\ar@{-->}@/^.5pc/[dddl]^{R^N(h)}&L(M)\ar@{-->}[ddd]^{\hat h}\ar[dr]^{L^N(f)=L(f)}\\
     &N&&&L(N)\\\\
     &M'\ar[uu]_{R^N(g)}&&K\ar[uur]^g\\
  }\]
  Namely, take $\eta$ to be the lattice-preserving map that acts as $f$ on elements when ignoring indices. In the left diagram, the upper and right triangles then commute by construction. By the left triangle, the map $h$ acts as $f$ on elements when ignoring indices. Because both paths along the outer triangle act as $f$ on elements, and both act as $h$ on the lattice, the outer triangle commutes. Since $\eta$ is lattice-preserving, there can exist at most one $\hat h$ that makes the large triangle commute; namely $L(h)$.

  Finally, to see that $R^N$ is full: within each rank-1 flat, the strong maps forming the objects of $\Matrp/N$ are one-to-one. Therefore, fixing the lattice maps that form the objects of $\GLat/L(N)$ constrains the morphisms of $\Matrp/N$ to identities on elements. 
  \qed
\end{proof}

\begin{corollary}
  Any orthogonal factorisation system $(\mathcal{L},\mathcal{R})$ in $\GLat$ induces an orthogonal factorization system $(L^{-1}(\mathcal{L}),\mathcal{R}')$ in $\Matrp$.
\end{corollary}
\begin{proof}
  Apply~\cite{zangurashvili2004several} to the previous proposition.
  \qed
\end{proof}

\subsection{Vector spaces}

Next we extend Example~\ref{ex:vectorspace} to a functor from the category $\FVect_k$ of finite vector spaces over $k$ and linear maps to the category of matroids.

\begin{proposition}\label{prop:functorMat}
  There is a functor $\Mat \colon \FVect_k \to \Matrp$ that sends a vector space $V$ to the matroid with ground set $V$ whose independent sets are the linearly independent subsets, acting on morphisms as $\Mat(f)=f$.
\end{proposition}
\begin{proof}
  Since flats correspond to vector subspaces, $\Mat(f)$ is indeed a strong map as the inverse image of a vector subspace is a vector subspace.
  \qed
\end{proof}

The functor $\Mat$ is faithful, but not full: for $V=k=\mathbb{Z}_4$, the function $f \colon V \to V$ given by $f(0)=0$, $f(1)=2$, $f(2)=1$, $f(3)=3$ is a strong map $\Mat(V) \to \Mat(V)$ but not linear.

\begin{lemma}
  The functor $\Mat$ does not preserve coproducts, so has no right adjoint.
\end{lemma}
\begin{proof}
  As we saw in Section~\ref{sec:limits} above, coproducts of matroids have to satisfy $|M + N| = |M| + |N|$, but $|\Mat(V \oplus W)| = V \oplus W \neq V \sqcup W = |\Mat(V) + \Mat(W)|$. 
  \qed
\end{proof}

A matroid $M$ is \emph{representable} (over $k$) if there is a strong map $f \colon M \to \Mat(V)$ for some vector space $V$ (over $k$) such that a subset $X\subseteq |M|$ is independent if and only if $f(X)$ is independent. 
In particular, to a matrix with entries in $k$ we may associate a matroid whose ground set consists of the columns of the matrix, and where a subset is independent if the corresponding columns are linearly independent. The rank function of the matroid counts the rank of the matrix of selected columns. 
Every representable matroid arises this way.
When the matroid $N$ is represented by a matrix $A$, and $B\subseteq |M|$, we write $A_{[B]}$ for the ordered  set of columns of $A$ labelled by elements of $B$.
Not all matroids are representable, so the functor $\Mat$ is not surjective on objects.
Nor is it injective on objects: swapping the role of two collinear elements in a vector space results in the same matroid. 

We now embark on altering the functor $\Mat$ to make it encompass all matroid representation. Intuitively, we consider the above way to turn a matrix into a matroid, and remove some structure from the matrix. 

\begin{lemma}\label{lem:matrixrank}
  If matrices $A,B,C$ satisfy $A=CB$, and $I\subseteq J$ are sequences of columns of $B$, and $I'\subseteq J'$ are the corresponding sequences of columns of $A$, then 
  \[
    \rank(A_{[J']})-\rank(A_{[I']})\leq\rank(B_{[J]})-\rank(B_{[I]}).
  \]
\end{lemma}
\begin{proof}
  Take any column $j\in J \setminus I$ and the corresponding $j'\in J' \setminus I'$. 
  Comparing $r=\rank(B_{[I]})-\rank(A_{[I']})$ with $r'=\rank(B_{[I\cup\{j\}]})-\rank(A_{[I'\cup\{i'\}]})$ gives three cases:
  \begin{itemize}
    \item$\rank(B_{[I\cup\{j\}]})=\rank(B_{[I]})$ and $\rank(A_{[I'\cup\{j'\}]})=\rank(A_{[I']})$, so $r'=r$.
    \item$\rank(B_{[I\cup\{j\}]})=\rank(B_{[I]})+1$ and $\rank(A_{[I'\cup\{j'\}]})=\rank(A_{[I']})$, so $r'=r+1$.
    \item$\rank(B_{[I\cup\{j\}]})=\rank(B_{[I]})+1$ and $\rank(A_{[I'\cup\{j'\}]})=\rank(A_{[I']})+1$, so $r'=r$.
  \end{itemize}
  In all cases $r\leq r'$, and so $\rank(A_{[I'\cup\{j'\}]})-\rank(A_{[I']})\le\rank(B_{[I\cup\{j\}]})-\rank(B_{[I]})$. The proof is completed by repeating for the other elements of $J \setminus I$.
  \qed
\end{proof}

We will consider matrices as multisets of vectors. Recall that a \emph{multisubset} of a set $S$ is a function $j \colon S \to \mathbb{N}$, which is \emph{finite} when $\supp(j)=j^{-1}(\mathbb{N} \setminus \{0\})$ is finite; a map between multisubsets $j \to j'$ is a function $\supp(j) \to \supp(j')$.

\begin{definition}\label{def:MVect}
  Write $\MVect_k$ for the category whose objects are finite multisubsets $j \colon V \to \mathbb{N}$ of some vector space $V$ over $k$, and whose morphisms $(V,j) \to (V',j')$ are linear maps $V \to V'$ that restrict to $\supp(j) \to \supp(j')$.
\end{definition}

There is a canonical inclusion $\cat{Vect}_k \to \MVect_k$.

\begin{theorem}\label{thm:M:extension}
  There is a functor $\Mat \colon \MVect_k \to \Matr$ sending $(V,j)$ to the matroid with ground set having an element for each element of $\supp(j)$ where a subset is independent if and only if the corresponding multisubset of vectors in $V$ is (a subset and) linearly independent.
  It makes the following diagram commute:
  \[\xymatrix@R-5ex{
    \FVect_k \ar[dd] \ar[dr]^-{\Mat}\\
    & \Matr \\
    \MVect_k \ar[ur]_-{\Mat}
  }\]
  It is the left Kan extension of $\Mat \colon \cat{Vect}_k \to \Matr$ along $\cat{Vect}_k \to \MVect_k$.
\end{theorem}
\begin{proof}
  It is clear that $\Mat \colon \MVect_k \to \Matr$     is well-defined on objects and functorial. 
  By construction the diagram commutes.
  To see that it is well-defined on morphisms, consider a morphism $(V,j) \to (V',j')$ of $\MVect_k$. Write $S=\supp(j)$, $S'=\supp(j')$, and denote the restriction by $f \colon S \to S'$. 
  The matroid $\Mat(f(S))$ has the same rank function as $\Mat(S')$, because both count the rank of submatrices of $f(S)$ and $S'$, which are identical apart from possibly repeated columns. Lemma~\ref{lem:latticemaps}(b) implies that this map $\Mat(f(S)) \to \Mat(S')$ is strong. Therefore it suffices to prove that $\Mat(S) \to \Mat(f(S))$ is strong.

  Choose bases for $V,V'$, so we may regard all of $f,V,V',S,S'$ as matrices.
  Let $I \subseteq J \subseteq |\Mat(S)|$. By Lemma~\ref{lem:matrixrank} then
  \begin{align*}
    & \rank_{\Mat(S')}\big(\Mat(f)(J)\big) - \rank_{\Mat(S')}\big(\Mat(f)(I)\big) \\
    & = \rank_{\Mat(f(S))}\big(\Mat(f)(J)\big) - \rank_{\Mat(f(S))}\big(\Mat(f)(I)\big) \\
        & = \rank\big((f(S))_{[\Mat(f)(J)]}\big) - \rank\big((f(S))_{[\Mat(f)(I)]}\big)\\
    & \leq \rank(S_{[J]}) - \rank(S_{[I]}) \\
    & = \rank_{\Mat(S)}(J) - \rank_{\Mat(S)}(I).
  \end{align*}
  Hence $\rank_{\Mat(S')}\big(\Mat(f)(J)\big) - \rank_{\Mat(S')}\big(\Mat(f)(I)\big) \leq \rank_{\Mat(S)}(J) - \rank_{\Mat(S)}(I)$, which by Lemma~\ref{lem:latticemaps}(b) implies that $\Mat(f)$ is strong.

  Finally, to see that this functor $\overline{\Mat} \colon \MVect_k \to \Matr$ is a left Kan extension of $\Mat \colon \cat{Vect}_k \to \Matr$ along $I \colon \cat{Vect}_k \to \MVect_k$, observe that for every $\Mat' \colon \MVect_k\to\Matr$ and natural transformation $\alpha \colon \Mat \Rightarrow \Mat' \circ I$ there is a unique natural transformation $\beta \colon \overline{\Mat} \Rightarrow \Mat'$ such that $\beta I = \alpha$. Uniqueness follows from the fact that any $\beta$ satisfying $\beta_2\circ \overline{\Mat}(g) = \Mat'(g) \circ \beta_1$ for all $g \colon 1\to2$ in $\MVect_k$ must act as the identity on elements (regarding repetitions as the same element). To show existence, it is easy to check that this $\beta$ is a strong map.
  \qed
\end{proof}

The rest of this subsection considers the functor $L \circ \Mat$ that turns a vector space into its lattice of vector subspaces, which is of interest to \textit{e.g.}\ quantum logic.

\begin{proposition}\label{prop:LMnearlyfaithful}
  The functor $L\circ M \colon \FVect_k \to \GLat$ is nearly faithful.
\end{proposition}
\begin{proof}
  It suffices to show that two linear maps $f,g \colon V \to W$ give rise to the same lattice morphism $L(\Mat(f)) = L(\Mat(g))$ if and only if they are multiples of each other.
  For all flats $X$ of $\Mat(V)$ we have $\clos\big(\Mat(f)(X)\big)=\clos\big(\Mat(g)(X)\big)$. 
  It follows from faithfulness of $M$ that $f(X)=g(X)$ for every subspace $X$ of $V$. 
  Therefore $f(v)=w$ implies $g(v)=\beta_vw$ for some $\beta_v \in k$. 
  For $v,v' \in V$ we have $f(v-v')=g(\beta v-\beta' v')$ for some $\beta,\beta' \in k$. So $f(v-v')=\beta' g(v-v')+(\beta-\beta')g(v)$. Choosing $v,v'$ that are not multiples of each other, $v-v'$ cannot be a multiple of $v$, so we must have $\beta=\beta'$, and $f=\beta g$.
  \qed
\end{proof}

The functor $L\circ \Mat$ is not full. To see this, consider the $\mathbb{Z}_2$-vector space $V=\mathbb{Z}_2^2$, which has three 1-dimensional subspaces. Let $A=\{(0,0),(1,0)\}$, $B=\{(0,0),(0,1)\}$ and $C=\{(0,0),(1,1)\}$, and consider the lattice map $g \colon L(\Mat(V))\to L(\Mat(V))$ that sends all atoms to $B$.
There is only one assignment $f \colon V \to V$ with $L(f)=g$, namely $f(0,0)=(0,0)$, $f(1,0)=(0,1)$, $f(0,1)=(0,1)$ and $f(1,1)=(0,1)$, but this function is not linear.
Even if we allow infinite base fields, $L \circ M$ is not full. 
For example, take $V=\mathbb{R}^2$, map all the lines through the origin to one of the axes, and the origin to itself. This preserves the loop, atoms, and joins, because everything is mapped to the same subspace. But there is no linear map implementing this assignment.

\subsection{Graphs}\label{subsec:graphs}

Lastly we briefly discuss functoriality of the construction of Example~\ref{ex:graph} turning an undirected graph into a matroid. We simplify the definition of undirected graph~\cite{brown2008graphs}, as we need not to distinguish bands and loops.

\begin{definition}\label{def:graph}
  Write $\cat{Graph}$ for the following category. Objects are undirected multigraphs: a set $V$ of vertices, a set $E$ of edges and a boundary map $\theta$ from $E$ to the class of singleton and two-element subsets of $V$.
  A morphism is a pair of maps $f \colon V \to V'$ and $g \colon E \to E'$ satisfying $f \circ \theta = \theta' \circ g$, and $\card(\theta'(e))\leq \card(\theta(e))$.
\end{definition}

To extend Example~\ref{ex:graph} to a functor $\cat{Graph} \to \Matr$, we could restrict the category of graphs to only permit `strong' morphisms of graphs, whose preimage preserves closed sets (here a set of edges is closed if the addition of an edge does not change the size of a spanning tree in the corresponding subgraph). There is some evidence that this choice of morphisms is useful for some applications of graph theory~\cite{recski1987elementary}.
Alternatively, we could allow more functions than strong maps as morphisms between matroids. We must at least keep the restriction that loops map to loops.
Or we could restrict both the domain and codomain.
Let us write $\cat{Graph}^*$ and $\Matr^*$ for the chosen domain and codomain.

For a functor $\Mat \colon \cat{Graph}^* \to \Matr^*$ to be of any practical use it should act as the identity on morphisms. It must then have the following properties:
\begin{itemize}
  \item It cannot be surjective on objects. A matroid is of the form $\Mat(G)$ precisely when it is \emph{graphic}, and there exist non-graphic matroids.

  \item It cannot be injective on objects.\footnote{We may actually make the functor injective by altering the graph $G$: Form the graph $G'$ by adding to $G$ an extra vertex $v$ as well as one edge $vw$ for every pre-existing vertex $w$; take $M(G)$ to be the cycle matroid of $G'$.} Here are graphs $G_1 \not\simeq G_2$ with $\Mat(G_1) = \Mat(G_2)$:
    
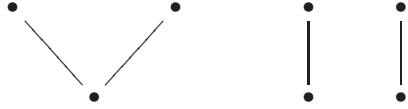
\[\xymatrix@C-1ex{
      \bullet && \bullet \\
      & \bullet \ar@{-}[ur] \ar@{-}[ul]
    } \qquad\qquad \xymatrix{
      \bullet & \bullet \\
      \bullet \ar@{-}[u] & \bullet \ar@{-}[u]
    }\]

  \item It cannot be full. There are no maps $G_1\to G_2$ that are surjective on edges, whereas $M(G_1)$ must have at least one morphism to itself (namely the identity).

  \item It cannot be faithful. Functions $G_2\to G_1$ corresponding to the identity matroid map may act differently on vertices.
\end{itemize}

One could define functors from the category of graphs and strong maps to the category of matroids that assign more obscure matroids to graphs, that we briefly list, but none of them is surjective or injective on objects, nor full or faithful.
\emph{Bicircular matroids}~\cite[Section 12.1]{oxley:matroids}, \emph{frame matroids} and \emph{lift matroids}~\cite{zaslavsky:biasedgraphs} reduce to the cycle matroid for the two graphs given above. \emph{Bond matroids}~\cite[Section 2.3]{oxley:matroids}, in the case of planar graphs, reduce to the cycle matroid of the dual graph. A \emph{transversal matroid}~\cite[Section 1.6]{oxley:matroids} is defined on the vertices of one side of a bipartite graph.

Recall that the coproduct in $\cat{Graph}$ is given by disjoint union. Taking strong morphisms in does not change this, that is, the inclusion $\cat{Graph}^* \to \cat{Graph}$ preserves and reflects coproducts. Therefore the functor $\Mat \colon \cat{Graph}^* \to \Matr$ must preserve coproducts.

\section{Constructions}\label{sec:constructions}

This section examines functoriality of various operations of matroid theory. 

We start with one of the most prominent ones: the \emph{dual} $M^*$ of a matroid $M$ has the same ground set, but the bases of $M^*$ are complements of bases of $M$.
One might hope that taking duals is functorial on matroids and strong maps, but that is not the case.

\begin{remark}
  There are matroids $M,N$ with strong maps $M \to N$ but no strong maps $M^* \to N^*$ or $N^* \to M^*$, so taking duals is not functorial on $\Matr$.
  (To see this, take $\mathcal{F}_M=\big\{\emptyset,\{a,b,c\}\big\}$ and $\mathcal{F}_N=\big\{\{*\}\big\}$; then there is only one map $M\to N$ and no maps $N\to M$, but there are maps both $M^*\to N^*$ and $N^*\to M^*$.)

  Taking duals is functorial on the subcategory of $\Matr$ of surjective strong maps between matroids of equal cardinality, since a quotient $q \colon M \to N$ does induce a strong map $q^* \colon N^*\to M^*$.
\end{remark}

We already saw in Remark~\ref{rem:reviewerscomment} that adding loops is a functorial process. We now prove that the same holds for adding isthmuses.

\begin{proposition}\label{prop:isthmus}
  There is an endofunctor $A \colon \Matr\to\Matr$ that acts on objects as $\mathcal{B}_{A(M)}=\{ B \cup \{*\} \mid B \in \mathcal{B}_M\}$, and on morphisms as $A(f)(*)=*$ and $A(f)(x)=f(x)$.
\end{proposition}
\begin{proof}
  By construction $*$ is an isthmus in $A(M)$.
  Since $A=(-)+D$, where $D$ is the free matroid on $\{*\}$, the assignment $A$ is clearly functorial. 
  \qed
\end{proof}

Next we consider the operations of deletion and contraction. As defined in Definition~\ref{def:deletion}, they are not functorial.
To see this, suppose $f \colon M\to N$ maps $m\mapsto n$ and $g\colon L\to M$ maps $l\mapsto m$, where $l,m,n$ are all nonloops; if $m$ is among the elements by which we contract or delete but the elements $l$ and $n$ are not, then the composite morphism $f\circ g$ cannot canonically be mapped to any strong map, either covariant or contravariant.
However, these operations become functorial when we change the category to ensure the deleted/contracted elements in $M$ are exactly those that map to deleted/contracted elements of $N$.  

\begin{proposition}
  Write $\Matr_*$ for the category whose objects are pairs $(M,Z)$ of $M \in \Matr$ and $Z \subseteq |M|$, and whose morphisms $(M,Z)\to(M',Z')$ are strong maps $M\to M'$ where $Z$ is the preimage of $Z'$. There are functors:
  \begin{align*}
    C \colon \Matr_* & \to \Matr     &   D \colon \Matr_* & \to \Matr \\
    (M,Z) & \mapsto M/Z & (M,Z) & \mapsto M\backslash Z \\
    f & \mapsto f & f & \mapsto f
  \end{align*}
\end{proposition}
\begin{proof}
  The operations on objects are clearly well-defined, and identity and composition are clearly respected. We need only show that the operations on morphisms are well-defined. For $X\subseteq Y\subseteq |M| \setminus Z$:
  \begin{align*}
    \rank_{C(M')}(f(Y)) - \rank_{C(M')}(f(X)) 
    & = \rank_{M'}(f(Y) \cup Z')-\rank_{M'}(f(X)\cup Z') \\
    & = \rank_{M'}(f(Y \cup Z)) - \rank_{M'}(f(X \cup Z)) \\   
    & \leq \rank_M(Y\cup Z) - \rank_M(X \cup Z)\\
    & = \rank_{C(M)}(Y) - \rank_{C(M)}(X), \\
  \intertext{and}
    \rank_{D(M')} (f(Y)) - \rank_{D(M')}(f(X))
    & = \rank_{M'}(f(Y)) - \rank_{M'}(f(X)) \\
    & \leq \rank_M(Y) - \rank_M(X) \\
    & =\rank_{D(M)}(Y) -\rank_{D(M)}(X).
  \end{align*}
  The result now follows from Lemma~\ref{lem:latticemaps}.
  \qed
\end{proof}

We may implement a series of $n$ deletions and contractions by employing the category $\Matr_{*n}$, whose objects are $(M, Z_1,\ldots,Z_n)$ where the sets $Z_i\subseteq |M|$ are disjoint, and whose morphisms $(M,Z_1,\ldots,Z_n)\to(M',Z'_1,\ldots,Z'_n)$ are strong maps $f \colon M\to M'$ such that $Z_i$ is the preimage of $Z'_i$ for all $i$. Then we can define functors $C,D \colon \Matr_{*n+1}\to\Matr_{*n}$. The composition of all these functors produces a minor, so taking minors in $\Matr$ is functorial.

\begin{theorem}
  Deletion $D \colon \Matr_{*n+1}\to\Matr_{*n}$ is right adjoint to the inclusion $\Matr_{*n}\to\Matr_{*n+1}$ given by  $(M,Z_1,\ldots,Z_n)\mapsto(M,Z_1,\ldots,Z_n,\emptyset)$ and $f \mapsto f$.
\end{theorem}
\begin{proof}
  Take the unit $\eta$ to be the identity, and the transpose $\hat{f} \colon (M,Z_1,\ldots,Z_n,\emptyset) \to (M',Z_1',\ldots,Z'_{n+1})$ to have the same underlying function as $f \colon (M,Z_1,\ldots,Z_n) \to (M',Z_1',\ldots,Z'_n)$.
  \qed
\end{proof}

The following matroid operation of extension turns out not to be functorial.

\begin{definition}[\cite{white:matroids}, Proposition 7.3.3]
  The \emph{free extension} of a matroid $M$ by $p$ is defined as the matroid $X(M)$ with $|X(M)|=|M|\cup\{p\}$ and flats
  \[
    \big \{ K \in\mathcal{F}_M \setminus \{|M|\} \big \}
    \cup
    \big \{ K\cup\{p\} \mid K\in\mathcal{F}_M \setminus \mathcal{H}(M)\}.
  \]
\end{definition}

\begin{remark}\label{freeext}
  Let $M$ be the free matroid on $\{a,b\}$ and $N$ the free matroid on $\{a,b,c,d\}$. Let $f\colon M\to N$ be the strong map $f=\{a\mapsto a,b\mapsto b\}$. Then $\mathcal{F}_{X(M)}=\big\{\emptyset,\{a\},\{b\},\{p\},\{a,b,p\}\big\}$ and
  \begin{align*}
    \mathcal{F}_{X(N)}=\big\{&\emptyset,\{a\},\{b\},\{c\},\{d\},\{p\},\\&\{a,b\},\{a,c\},\{a,d\},\{b,c\},\{b,d\},\{c,d\},\{a,p\},\{b,p\},\{c,p\},\{d,p\},\\&\{a,b,c\},\{a,b,d\},\{a,b,p\},\{a,c,d\},\{a,c,p\},\\&\{a,d,p\},\{b,c,d\},\{b,c,p\},\{b,d,p\},\{c,d,p\},\\&\{a,b,c,d,p\}\big\}.
  \end{align*}
  There are no strong maps $X(M)\to X(N)$ that agree with $f$. Hence $X(f)$ cannot be canonically defined in a way that respects identities, and the free extension cannot be functorial.
\end{remark}

It follows that the matroid operation of \emph{truncation} (contraction by $p$ after free extension by $p$) cannot be functorial, because free extension is equivalent to truncation by $p$ after the addition of an isthmus $p$, which is functorial by Proposition~\ref{prop:isthmus}.
A similar counterexample (omitted here) shows that the dual operations of the free extension and the truncation, namely the \emph{free coextension} and the Higgs \emph{lift}, are not functorial either.

Finally we consider \emph{erection}, the inverse matroid operation of truncation. 
The erections of a matroid $M$ form a lattice based on a certain ordering \cite[Chapter~7.5]{white:matroids}, whose top element $E(M)$ is the so-called \emph{free erection} and whose bottom element is $M$ itself; by convention $E(M)=M$ if $M$ has no erections at all.
This operation turns out to be not functorial.

\begin{remark}
  Let $M,N$ be the erectible matroids with the following sets of flats:
  \begin{align*}
    \mathcal{F}(M)
    = \big \{&\emptyset,\\
    &\{0\},\{1\},\{2\},\{3\},\{4\},\\
    &\{0,1\},\{0,2\},\{0,3\},\{0,4\},\{1,2\},\{1,3\},\{1,4\},\{2,3,4\},\\
    &\{0,1,2,3,4\}\big\}\\
    \mathcal{F}(N)=\big\{&\emptyset,\{0\},\{1\},\{2\},\{3\},\{4\},\{0,1,2,3,4\}\big\}.
  \end{align*}
  Then the identity function is a strong map $f \colon M\to N$.
  But there are no strong maps $E(M)\to E(N)$ that agree with $f$, so erection cannot be functorial.
\end{remark}

So far we have considered operations that input a single matroid. The rest of this section considers operations that combine two matroids into a new one, and discusses whether they give monoidal structure on the category $\Matr$ of $\Matrp$.
For example, by Proposition~\ref{prop:coproducts}, the coproduct gives one monoidal structure.

The following operations from matroid theory do not give monoidal structure:
the \emph{sum} or \emph{union} $M \cup N$ is the matroid whose independent sets are the unions of the independent sets of its constituents;
the \emph{product} or \emph{intersection} of $M \cap N$ is the matroid $(M^* \cup N^*)^*$;
the \emph{half-dual sum} is the matroid $M^* \cup N^*$.

\begin{remark}
  Unions do not give monoidal structure: there are matroids $A,B,C,D$ and maps $A\to C$ and $B\to D$ such that there are no maps $A\cup B\to C\cup D$. 
  \begin{align*}
    \mathcal{F}_A = & \; \mathcal{F}_B = \big\{\{\bullet\},\{\bullet,0,1,2\}\big\}, \\
    \mathcal{F}_C = & \; \big \{\{f\},\{f,a\},\{f,b\},\{f,c\},\{f,d\},\{f,e\},\\
    &\;\{f,a,b\},\{f,a,c\},\{f,a,d\},\{f,a,e\},\{f,e,b\},\{f,e,c\},\{f,e,d\},\\
    &\;\{f,a,e,b\},\{f,a,e,c\},\{f,a,e,d\},\{f,a,e,b,c,d\}\big\}, \\
    \mathcal{F}_D=& \; \big\{\{e\},\{e,a\},\{e,b\},\{e,c\},\{e,d\},\{e,f\},\\
    &\;\{e,a,b\},\{e,a,c\},\{e,a,d\},\{e,a,f\},\{e,b,c\},\\
    &\;\{e,b,d\},\{e,b,f\},\{e,c,d\},\{e,c,f\},\{e,d,f\},\\
    &\;\{e,f,a,b\},\{e,f,a,c\},\{e,f,a,d\},\{e,f,b,c\},\{e,f,b,d\},\{e,f,c,d\},\{e,a,b,c,d\},\\
    &\;\{e,a,b,c,d,f\}\big\}
  \end{align*}
  Then $\mathcal{F}_{A\cup B}=\big\{\{\bullet\},\{\bullet,0\},\{\bullet,1\},\{\bullet,2\},\{\bullet,0,1,2\}\big\}$ and $C\cup D$ is the free matroid on $\{a,b,c,d,e,f\}$. Hence $A\cup B$ has a loop but $C\cup D$ does not, so there can be no maps $A\cup B\to C\cup D$, whereas there are maps $A\to C$ and $B\to D$.
\end{remark}

\begin{remark}
  Intersections do not give monoidal structure: there are strong maps $A\to C$ and $B\to D$ such that there are no maps $A\cap B\to C\cap D$. 
  \[
    \mathcal{F}_A = \mathcal{F}_B = \big\{\emptyset,\{1,2\}\big\},
    \qquad
    \mathcal{F}_C = \mathcal{F}_D = \big\{\emptyset,\{x\}\big\}.
  \]
  Then $\mathcal{F}_{A\cap B}=\big\{\{1,2\}\big\}$ and $\mathcal{F}_{C\cap D}=\big\{\emptyset,\{x\}\big\}$. 
  Hence there is a map $1,2\mapsto x$ in each homset $\Matr(A,C)$ and $\Matr(B,D)$, but there can be no map $A\cap B\to C\cap D$ because there is no loop in $C\cap D$ to receive the loops of $A\cap B$.
\end{remark}

\begin{example}
  Half-dual unions do not give monoidal structure: there are maps $A\to C$ and $B\to D$ such that there are no maps $A\cup B^*\to C\cup D^*$. 
  \[
    \mathcal{F}_A=\big\{\{0\}\big\},
    \qquad 
    \mathcal{F}_B=\big\{\emptyset,\{0\}\big\}, 
    \qquad
    \mathcal{F}_C=\mathcal{F}_D=\big\{\{*\}\big\}.
  \]
  Then $\mathcal{F}_{A\cup B^*}=\big\{\{0\}\big\}$ and $\mathcal{F}_{C\cup D^*}=\big\{\emptyset,\{*\}\big\}$. Once again, the homsets $\Matr(A,C)$ and $\Matr(B,D)$ each have to contain at least the map sending everything to the loop, whereas $A\cup B^*$ has a loop that cannot be mapped to anything in $C\cup D^*$.
\end{example}

The \emph{intertwining} of two matroids, defined as a minor-minimal matroid that contains them both as minors, is not a monoidal product either, as it is not always unique up to isomorphism.

We end on a positive note, by showing that altering the category slightly allows for two more monoidal structures.
Write $\Matr_\times$ for the category of matroids with a distinguished element and strong maps preserving the distinguished element.
The \emph{parallel connection} $M || N$ is the coproduct in this category~\cite{white:matroids}.
Explicitly, the ground set of $M||N$ is the disjoint union of $|M|$ and $|N|$, the distinguished elements are identified, and the flats of $M||N$ are the unions of flats in $M$ and in $N$.
This is similar to the coproduct in $\Matr$, except that the distinguished element need not be a loop.

Finally, the \emph{series connection} $MN$ is defined dually to the parallel connection: $(M^*||N^*)^*$, and gives another monoidal structure on $\Matr_\times$. 
This monoidal structure is not a product, but remarkably enough it is naturally affine, in the sense that there are always natural transformations $MN\to M$ and $MN \to N$.
Neither of the parallel connection and the series connection distributes over each other.

\section{The greedy algorithm}\label{sec:greedy}

There exists a well-known characterization of matroids which, intriguingly, is algorithmic in nature and exemplifies the connection between matroids and problems in combinatorics~\cite{oxley:matroids,leelneedauer:greedy}.

\begin{definition}\label{optimization}
  Let $\I$ be a collection of subsets of a finite set $E$ that satisfies the nontrivial and downward closed conditions from Definition~\ref{def:matroid}. Given a function $w \colon E\to\mathbb{R}$ define the associated \emph{weight} function $w \colon 2^E\to\mathbb{R}$ by
  \[
    w(X)=\sum_{x\in X}w(x).
  \]
  The \emph{optimization problem} for the pair $(\I,w)$ is to find a maximal member $B$ of $\I$ of maximum weight.
\end{definition}

\begin{definition}
  The \emph{greedy algorithm} for a pair $(\I,w)$ as in Definition~\ref{optimization} is:
  \begin{enumerate}[label=(\roman*)]
  \item Set $X_0=\emptyset$ and $j=0$.
  \item\label{previousstep} If $E-X_j$ contains an element $e$ such that $X_j\cup\{e\}\in\I$, choose such an element $e_{j+1}$ of maximum weight, let $X_{j+1}=X_j\cup\{e_{j+1}\}$, and go to \ref{nextstep}; otherwise, let $B_G=X_j$ and go to \ref{laststep}.
  \item\label{nextstep} Add 1 to $j$ and go to \ref{previousstep}.
  \item\label{laststep} Stop.
  \end{enumerate}
\end{definition}

\begin{theorem}\label{greedy}
Let $\I$ be a nontrivial and downward closed collection of subsets of a finite set $E$. Then $\I$ is the collection of independent sets of a matroid on $E$ if and only if the greedy algorithm for $(\I,w)$ solves the optimization problem for $(\I,w)$ for all possible weight functions $w \colon E\to\mathbb{R}$ (generalized to $w \colon 2^E\to\mathbb{R}$ as in Definition \ref{optimization}).
\end{theorem}
\begin{proof}
  See~\cite[Theorem 1.8.5]{oxley:matroids}.
  \qed
\end{proof}

Crucially, this theorem is equivalent to the following statement: ``The greedy algorithm solves the optimization problem if and only if all maximal independent sets have the same cardinality''. It is easy to show that the latter condition is equivalent to the independence augmentation axiom of Definition~\ref{def:matroid} when the other two axioms are given.

We observe that this theorem induces an elegant categorical characterization of matroids. Write \cat{Vect^b_k} for the category of vector spaces over $k$ with a chosen basis $b$ and linear transformations between them.

\begin{lemma}
  Every run of the greedy algorithm produces a maximal chain of epimorphisms in a subcategory of \cat{Vect^b_\mathbb{R}}.
\end{lemma}
\begin{proof}
  Let the final output of the algorithm be the list $B=(b_1,b_2,\hdots,b_r)$. The vector $(w(b_1),w(b_2),\hdots,w(b_r))$ is an element of $\mathbb{R}^r$. At the $n$th step of the algorithm, the candidate output corresponds to a vector in $\mathbb{R}^n$. Then the $n$th step of the algorithm corresponds to an epimorphism $e_n \colon \mathbb{R}^n\to\mathbb{R}^{n-1}$ which projects out the largest element of the vector. The algorithm continues as long as there exist incoming morphisms in the subcategory formed by all such epimorphisms (\ie there exist candidate elements), making the chain maximal in that diagram.
\qed
\end{proof}

This is the categorical equivalent of the fact that the greedy algorithm produces a maximal independent set; the length of the chain equals the cardinality of that set. The following definition makes precise when a partially ordered set is `as wide as it is tall'.

\begin{definition}
  An element $z$ in a partially ordered set \emph{covers} $x$ when $x \leq z$, and if $x \leq y \leq z$ then $x=y$ or $y=z$.
  Define a \emph{square poset} to be a finite partially ordered set $P$ with least element, such that any element $A \in P$ covers exactly $n_A+1$ elements, where $n_A$ is the maximum length of a maximal chain from the least element to $A$.
  A \emph{square functor} is a functor $I \colon P \to \cat{Sub}$ from a square poset $P$ to the category $\cat{Sub}$ of sets and inclusions that is injective on objects and preserves chain lengths.
\end{definition}

Every square functor $I \colon P \to\cat{Sub}$ induces a pair $(\I,E)$ where $\I$ is a nontrivial downwards closed collection of subsets of a set $E$. Namely, define $E$ to be the union of all the sets $S_i$ in the image of $I$, and set $\I=\{S_i\setminus I(0)\}$, where $0$ is the least element of $P$. This is evidently a collection of subsets of $E$, and it contains the empty set. Because $I$ is injective on objects, the number of inclusions into each object $S_i$ is maximal, guaranteeing that all subsets of each member of $\I$ are in the image of $I$. 

\begin{lemma}
  Given a square functor $I \colon \cat{P}\to\cat{Sub}$, the induced pair $(\I,E)$ is the collection of independent sets and ground set of a matroid if and only if every maximal chain in the image of $I$ has the same length; equivalently, if the colimit of the diagram $I \colon C\to\cat{Sub}$ is independent, up to isomorphism, of a maximal chain $C \subseteq P$.
\end{lemma}
\begin{proof}
  Each nonidentity inclusion in $C$ adds one element to the domain, therefore the length of the chain equals the cardinality of the final codomain. This statement is therefore equivalent to ``all maximal elements of $\I$ have the same cardinality", which together with the first two axioms defines a matroid. The colimit formulation now follows from the fact that sets with the same cardinality are isomorphic.
  \qed
\end{proof}

\begin{lemma}\label{lastlemma}
  Given a square functor $I \colon P\to\cat{Sub}$, the induced pair $(\I,E)$ is the collection of independent sets and ground set of a matroid if and only if for every contravariant functor $W \colon P\to\cat{Vect^b_\mathbb{R}}$ such that $I$ factors through $W$, the limit of the diagram $W \colon C\to\cat{Vect^b_\mathbb{R}}$ is independent of the maximal chain $C \subseteq P$.
\end{lemma}
\begin{proof}
  This is equivalent to the above lemma, by the definition of a limit and the fact that vector spaces of the same dimension are isomorphic.
  \qed
\end{proof}

The following theorem summarizes the main result of this section.

\begin{theorem}
  The greedy algorithm solves the optimization problem if and only if the chains in \cat{Vect^b_\mathbb{R}} induced by all runs have the same limit.
\end{theorem}
\begin{proof}
  Combine Theorem~\ref{greedy} and Lemma~\ref{lastlemma}.
  \qed
\end{proof}

\begin{acknowledgements}
  We are indebted to an anonymous referee for insightful comments that improved this article.
  We thank Nathan Bowler and Jeffrey Giansiracusa for encouragement and valuable input. 
  Lastly, for this amended version of the paper we are very grateful to Joshua Mundinger for pointing out an error in a proof, resulting in the removal of one lemma and one theorem.
\end{acknowledgements}

\bibliographystyle{spmpsci}
\bibliography{matroids}

\end{document}